\documentclass[11 pt]{amsart}

\usepackage{amscd}
\usepackage{amsmath}
\usepackage{amsthm}
\usepackage{amsfonts}
\usepackage{amssymb}
\usepackage{amsxtra}

\def\H{\mathbb{H}}
\def\C{\mathbb{C}}
\def\Z{\mathbb{Z}}
\def\N{\mathbb{N}}

\def\R{\mathbb{R}}

 \newtheorem{thm}{Theorem}[section]
 
 \newtheorem{lem}[thm]{Lemma}
 \newtheorem{prop}[thm]{Proposition}

\newcommand{\be}{\begin{equation}}
\newcommand{\ee}{\end{equation}}
\newcommand{\bea}{\begin{eqnarray}}

\newcommand{\eea}{\end{eqnarray}}
\newcommand{\Bea}{\begin{eqnarray*}}
\newcommand{\Eea}{\end{eqnarray*}}

\newcounter{cnt1}
\newcounter{cnt2}
\newcounter{cnt3}
\newcommand{\blr}{\begin{list}{$($\roman{cnt1}$)$}
 {\usecounter{cnt1} \setlength{\topsep}{0pt}
 \setlength{\itemsep}{0pt}}}
\newcommand{\bla}{\begin{list}{$($\alph{cnt2}$)$}
 {\usecounter{cnt2} \setlength{\topsep}{0pt}
 \setlength{\itemsep}{0pt}}}
\newcommand{\bln}{\begin{list}{$($\arabic{cnt3}$)$}
 {\usecounter{cnt3} \setlength{\topsep}{0pt}
 \setlength{\itemsep}{0pt}}}
\newcommand{\el}{\end{list}}
\title[Weyl multipliers]{Weighted norm inequalities for Weyl\\ multipliers and Fourier multipliers on the Heisenberg group}
\author{Sayan Bagchi and Sundaram Thangavelu}
\date{}
\begin{document}


\address{department of mathematics,
Indian Institute of Science, Bangalore - 560 012, India}
\email{sayan@math.iisc.ernet.in} \email{veluma@math.iisc.ernet.in}


\keywords{ Heisenberg group, Weyl transform, Weyl multipliers,
weighted norm inequalities, R-boundedness,  $L^p$ multipliers on
the Heisenberg group.} \subjclass[2010] {Primary:  43A80, 42B25.
Secondary: 42B20, 42B35, 33C45.}


\begin{abstract}
In this paper we prove weighted norm inequalities for Weyl
multipliers satisfying Mauceri's condition. As  an application, we prove certain multiplier theorems on the
Heisenberg group and also show in the context of a theorem of Weis
on operator valued Fourier multipliers that the  R-boundedness of
the derivative of the multiplier is not necessary for the
boundedness of the multiplier transform.

\end{abstract}


\maketitle

\section [Introduction]
{Introduction and the main results}
 In this paper we are concerned with weighted norm inequalities for
 Weyl multipliers and the relevance of such inequalities in the study of Fourier multipliers on the Heisenberg group.
 Building upon a result of Mauceri \cite{M} we prove certain weighted norm inequalities and then investigate the possibility of using
 them to prove boundedness of Fourier multipliers on the Heisenberg group.
 Weyl multipliers naturally occur in the context of  Fourier multipliers on the Heisenberg group if we view the
 latter as operator-valued multipliers for the Euclidean Fourier transform.
 In this context, there is an interesting result of L. Weis \cite{W} which gives a sufficient condition on the (operator-valued) multipliers,
 so that such Fourier multipliers are bounded on $ L^p $ spaces of Banach space valued functions.
 However, our investigations have led to the conclusion that one of the conditions
 in the above mentioned theorem of Weis is not necessary for the boundedness of the multiplier transform.
 Nevertheless, we prove some versions of multiplier theorems on the Heisenberg group. \\

In order to set-up notation and state our main results, we begin with recalling some basic definitions. Consider the Euclidean Fourier transform defined on $ L^1(\R^n)$
by
$$\hat{f}(\xi)=(2\pi)^{-\frac{n}{2}}\int_{\R^n}e^{-ix.\xi}f(x)dx.$$
 It is well known that
the map $f\rightarrow \hat{f}$ extends to the whole of $L^2(\R^n)$
as a unitary operator. Given a bounded measurable function
$m(\xi)$ on $\R^n$ we can define a transformation $T_m$ by setting
$$\widehat{(T_mf)}(\xi)=m(\xi)\hat{f}(\xi),\;\;\;\;\;\;f\in L^2(\R^n).$$
It is clear that $T_m$ is a bounded operator on $L^2(\R^n)$ but
without further assumptions it need not extend to $L^p(\R^n)$
 as a bounded operator for $p \neq 2$.
When it extends we say that $m$ (or equivalently $T_m$) is a
Fourier multiplier for $L^p(\R^n)$. Some sufficient conditions are
provided by H\={o}rmander-Mihlin and Marcinkiewicz multiplier
theorems, see \cite{D}. For instance, when $n=1$, the boundedness
of $m(\xi)$ together with that of $\xi m'(\xi)$ is sufficient for
the
boundedness of $T_m$ on $L^p(\R)$ for all $1<p<\infty$.\\

In the
non-commutative set-up we have an analogue of the Fourier
transform, namely the Weyl transform, which shares many important
properties with the Fourier transform. As is well known, this transform is closely related to the  Fourier transform on the Heisenberg group $ \H^n.$ For $f\in L^1\cap
L^2(\C^n)$, its Weyl transform $W(f)$ is defined as an operator on $ L^2(\R^n) $ by the equation
$$W(f)=\int_{\C^n} f(z) W(z)dz$$
where $W(z):L^2(\R^n)\rightarrow L^2(\R^n)$ is the unitary
transformation given by
$$W(z)\varphi(\xi)=e^{i(x.\xi+\frac{1}{2}x.y)}\varphi (\xi+y),\;\;\;\;\; \varphi\in L^2(\R^n).$$
It is known that $W$ takes $L^2(\C^n)$ onto the space of the
Hilbert-Schmidt operators on $L^2(\R^n)$. Analogous to  Fourier
multipliers one can define Weyl multipliers as follows: given a
bounded linear operator $m$ on $L^2(\R^n)$ we can define an
operator $T_m$ on $L^2(\C^n)$ by
$$W(T_mf)=mW(f)$$
which is certainly bounded on $L^2(\C^n)$. If this operator
extends to a bounded linear operator on $L^p(\C^n)$ then we say
that $m$ is a (left) Weyl multiplier for $L^p(\C^n)$. We can also
define right Weyl multipliers. \\

 In  \cite{M} Mauceri has obtained sufficient conditions on a bounded linear
 operator $m$ on $L^2(\R^n)$ so that the Weyl multiplier $T_m$
 is  bounded on $L^p(\C^n)$. In order to state this result we
 need to introduce some notation. The  spectral decomposition of the Hermite operator is given
 by
 $$H = -\Delta+|x|^2 =\sum^\infty_{j=0}(2j+n)P_j$$
 where $P_j$ are the projections onto the eigenspaces corresponding to the eigenvalues $ (2j+n) $
of the Hermite operator. We can decompose H as
 $$H=\frac{1}{2} \sum_{j=1}^n (A_jA_j^*+A_j^* A_j)$$
 where $A_j=\frac{\partial}{\partial x_j}+x_j$ and $A_j^*=-\frac{\partial}{\partial
 x_j}+x_j$ are the annihilation and creation operators. In terms
 of $A_j$ and $A_j^*$ we define noncommutative derivations
 $$\delta_j m=[m, A_j],~~~~~ \bar{\delta}m=[A_j^*,m].$$
 For multi-indices $\alpha,\beta\in \N^n$ we define
 $$\delta^\alpha=\delta_1^{\alpha_1}\delta_2^{\alpha_2}...\delta_n^{\alpha_n},
~~~~
\bar{\delta}^\beta=\bar{\delta_1}^{\beta_1}\bar{\delta_2}^{\beta_2}...\bar{\delta_n}^{\beta_n}.$$
We say that an operator $S\in B(L^2(\R^n))$ is of class $C^k$ if
$\delta^\alpha \bar{\delta}^\beta S\in B(L^2(\R^n))$ for all
$\alpha, \beta \in \N^n$ such that $|\alpha|+|\beta|\leq k$. We
also define
 $$\chi_k= \sum_{2^{k-1}\leq2j+n<2^k}P_j$$
 The following theorem has been proved in Mauceri \cite{M}.\\

\begin{thm} (\textbf{Mauceri})
  Let $m \in B(L^2(\R^n))$ be an operator of class $C^{n+1}$ which satisfies the following conditions: For all
 $\alpha, \beta\in N^n, |\alpha|+|\beta|\leq n+1$
\be\sup_{k\in \N^n} 2^{k(|\alpha|+\beta|-n)}||(\delta^\alpha
\bar{\delta}^\beta m)\chi_k||^2_{HS}\leq C~.\ee
Then the Weyl multiplier $T_m$ is bounded on $L^p(\C^n),
 1<p\leq 2.$ If the above assumption is replaced by
 \be\sup_{k\in \N^n} 2^{k(|\alpha|+\beta|-n)}||\chi_k(\delta^\alpha \bar{\delta}^\beta m)||^2_{HS}\leq
 C~.\ee
 then $T_m$ is bounded on $L^p(\C^n)$, $2\leq p<\infty$.
\end{thm}

In \cite{M} Mauceri has obtained good estimates  on the kernels associated to Weyl multipliers. It turns out that with a bit more effort we can do better than this.\\

\begin{thm} Let $m \in B(L^2(\R^n))$ satisfy the condition (1.1) for all
 $\alpha, \beta\in N^n, |\alpha|+|\beta|\leq n+2$.
Then the operator $T_m$ is bounded on $L^p(\C^n)$, $1<p<\infty$.
Moreover, for all $w\in A_{p/2}(\C^n), 2<p<\infty,$ it also satisfies the weighted norm inequality
$$\int_{\C^n}|T_m f(z)|^p w(z)dz\leq C~ \int_{\C^n}|f(z)|^pw(z)dz$$
for all $ f \in L^p(\C^n, w).$
\end{thm}

Note that the weight function $w$ is taken from $A_{p/2}$, not
from $A_p$ as one would expect. If we increase the number of non-commutative derivatives from $n+2$ to $2n+2$ then we can allow
$A_p$ weights in the weighted norm
inequality.\\
\begin{thm} Let $m \in B(L^2(\R^n))$ satisfy the condition (1.1) for all
 $\alpha, \beta\in N^n, |\alpha|+|\beta|\leq 2(n+1)$. Then, for all $w\in A_{p}(\C^n), 1<p<\infty,$
 $$\int_{\C^n}|T_m f(z)|^p w(z)dz\leq C~ \int_{\C^n}|f(z)|^pw(z)dz$$
for all $ f \in L^p(\C^n, w).$

\end{thm}

In the definition of the Weyl transform we have made use of the unitary operators $ W(z) $ acting on $ L^2(\R^n)$ and mentioned that these are related to certain representations of the Heisenberg group.  As a manifold $ \H^n = \C^n \times \R $ and  the group law on $H^n$
is given by $(z,t)(w,s)=(z+w, t+s+\frac{1}{2}\Im (z.\bar{w}))$.
For each $ \lambda \in \R\setminus \{0\}$ we have an irreducible
representation $\pi_\lambda$ of $H^n$ realised on $L^2(\R^n)$. The
explicit expression for $\pi_\lambda$ is
$$\pi_\lambda(z,t)\varphi(\xi)=e^{i\lambda t } e^{i\lambda (x.\xi+\frac{1}{2}x.y)}\varphi(\xi+y)$$
where $\varphi\in L^2(\R^n)$ and $z=x+iy$. It is clear that
$\pi_1(z,0)=W(z)$. Analogous to the Weyl transform we can also define the operators $ W_\lambda(f) $ by
$$ W_\lambda(f) = \int_{\C^n} f(z) \pi_\lambda(z,0) dz. $$ These are also called the Weyl transforms in the literature.\\

The Fourier transform of $f\in L^1\cap
L^2(H^n)$ is defined to be the operator valued function
$$\hat{f}(\lambda)=\int_{H^n} f(z,t)\pi_\lambda(z,t) dzdt.$$
It is known that for each $\lambda\in \R\setminus \{0\}$,
$\hat{f}(\lambda)$ is a Hilbert-Schmidt operator and we have
inversion and Plancherel theorems, see e.g \cite{T1}. In analogy
with Fourier  multipliers on $\R^n$ and Weyl multipliers on $\C^n$
we can define multipliers for the (group) Fourier transform on
$H^n$. Given a family of bounded linear operators $\{m(\lambda):
\lambda\in \R^*\}$ on $L^2(\R^n)$ we can define $T_m$ on $L^1\cap
L^2(H^n)$ by
$$\widehat{(T_m f)}(\lambda)=m(\lambda)\hat{f}(\lambda).$$
If the family $\{m(\lambda):\lambda\in \R^*\}$ are uniformly
bounded on $L^2(\R^n)$, it is clear (from Plancherel theorem) that
$T_m$ is bounded on $L^2(H^n)$. We are interested in finding
sufficient conditions on $m(\lambda)$ so that $T_m$ will extend to
$L^p(H^n)$ as a bounded
operator.\\

In this generality not much is known except for the results proved
in the papers by Mauceri-de Michele \cite{MM} (for $ n =1 $) and
Lin \cite{L}, for general $ \H^n.$ Lin has also looked at the
bounedness of $ T_m $ on Hardy spaces. In \cite{M2} Mauceri
studied zonal multipliers on the Heisenberg group.   When
$m(\lambda)= \varphi(H(\lambda))$, $H(\lambda)$ being the scaled
Hermite operator $(-\Delta+\lambda^2 |x|^2)$ on $\R^n$, the
operator $T_m $ becomes $\varphi(\mathcal{L})$, where
$\mathcal{L}$ is the sublaplacian on $H^n$(which plays the role of
$\Delta$ for $H^n$). There are several works on the $L^p$
boundedness of $\varphi(\mathcal{L})$ and
the best possible result has been obtained by  Muller-Stein \cite{MS} and Hebisch \cite{H}.\\

We now bring out the connection between Fourier multipliers on the
Heisenberg group and operator valued multipliers for the Euclidean
Fourier transform. Recalling the definition of $ \hat{f}(\lambda)
$ and noting that $ \pi_\lambda(z,t) = e^{i\lambda t} \pi_\lambda(z,0) $
we see that
$$ \hat{f}(\lambda) = \int_{\C^n} \left( \int_{-\infty}^\infty  f(z,t) e^{i\lambda t} dt \right) \pi_\lambda(z,0) dz.$$
Denoting the inner integral, which is the inverse Fourier
transform of $ f $ in the central variable, by $ f^\lambda(z) $ we
have $ \hat{f}(\lambda) = W_\lambda(f^\lambda).$ With this
notation, the Fourier multiplier $ T_m $ takes the form
$$ T_mf(z,t) = \int_{-\infty}^\infty e^{-i\lambda t} M(\lambda)f^\lambda(z) d\lambda $$
where the operator $ M(\lambda) $ is related to $ m(\lambda) $ by the equation
$$ W_\lambda(M(\lambda)f^\lambda) = m(\lambda)\hat{f}(\lambda) = m(\lambda)W_\lambda(f^\lambda) .$$
This means that $ M(\lambda)$ is a Weyl multiplier for each $\lambda\in \R^*$.\\

We can identify $ L^p(\H^n) $ with the space $ L^p(\R, L^p(\C^n)) $ consisting of all functions $ F $ on $ \R $
taking values in the Banach space $ L^p(\C^n) $ for which the function $ t \rightarrow  \|F(t)\|_{L^p(\C^n)} $ belongs to $ L^p(\R).$
The identification is given by the correspondence  $ F(t)(z) = f(z,t) $ for  $ f \in L^p(\H^n). $
With this identification, note that  the function $ f^\lambda \in L^p(\C^n) $ is nothing but the inverse Fourier transform of $F$:
$$ f^\lambda  = \int_{-\infty}^\infty  e^{i\lambda t} F(t) dt = \hat{F}(-\lambda)$$ where the integral is taken in the sense of Bochner.
Thus the action of the Heisenberg group Fourier multiplier $ T_m $ on $ f $ can be viewed as
$$ T_mf(\cdot,t) = \int_{-\infty}^\infty e^{i\lambda t} M(-\lambda)\hat{F}(\lambda) d\lambda.$$
This means that Fourier multipliers on the Heisenberg group can be viewed as
operator valued Euclidean Fourier multipliers acting on Banach space valued functions.\\

More generally, suppose $ X $ and $ Y $ are Banach spaces and $ \lambda \rightarrow M(\lambda) $ is a
function taking values in $ B(X,Y),$  the Banach space of bounded
linear operators from $ X $ into $ Y.$  Then  we can define
operator valued Fourier multipliers for functions taking values in
$ X$  by
$$ T_Mf(t) =  (2\pi)^{-1/2}  \int_{-\infty}^\infty  e^{-i\lambda t} M(\lambda)\hat{f}(\lambda)   d\lambda .$$
Then for the operator $ T_M $ to extend as a bounded linear
operator from $ L^p(\R, X) $ into $ L^p(\R,Y) $ the spaces $ X $
and $ Y $ have to be UMD spaces. Moreover, unlike the scalar
valued case, just the boundedness of the families $ \{ M(\lambda) : \lambda \in \R \}  $
and $ \{ \lambda M'(\lambda): \lambda \in \R \} $ is not enough.  In this
context L. Weis \cite{W} has  proved the following multiplier theorem for the Fourier transform.\\

\begin{thm} (\textbf{Weis}) Assume that $ X $ is UMD and $ M(\lambda) \in B(X,X)$ for each $ \lambda \in \R.$
Suppose $\{ M(\lambda):\lambda\in \R\}$ and $\{\lambda
M'(\lambda):\lambda\in \R\}$ are both R-bounded. Then the operator
valued Fourier multiplier $T_M$ defined  by
$$ T_Mf(t) = (2\pi)^{-1/2} \int_{-\infty}^\infty e^{-i\lambda t} M(\lambda)\hat{f}(\lambda) $$
extends to  $L^p(\R,X)$  as a bounded operator for all
$1<p<\infty$.
\end{thm}

The R-boundedness of  a family of operators $ \tau \subset B(X,X)
$ is defined using Rademacher functions $ r_j, j \in \N.$  For any
sequence $ M_j \in \tau $ and $ x_j \in X $ it is required that
there is a constant $ C $ such that
$$ \int_0^1 \| \sum_{j=1}^\infty r_j(u) M_jx_j \| du \leq C~ \int_0^1 \| \sum_{j=1}^\infty r_j(u) x_j \| du .$$
When $ X = L^p(\R^n) $ the R-boundedness is equivalent to the
vector valued inequality
$$ \| \left(\sum_{j=1}^\infty |M_jf_j|^2 \right)^{1/2} \|_p \leq C~ \| \left(\sum_{j=1}^\infty |f_j|^2 \right)^{1/2} \|_p $$ for all
choices of $ M_j \in \tau $ and $ f_j \in L^p(\R^n).$ Given a
family  $ \{  M(\lambda): \lambda \in \R \} $ of bounded linear operators acting on $
L^p(\R^n) $ it would be interesting to find some conditions which
will imply the R-boundedness. There are some special cases where
we do have
such conditions guaranteeing the R-boundedness.\\

Let $ H(\lambda) = -\Delta+\lambda^2 |x|^2 $ be the scaled Hermite
operator on $ \R^n $ whose  spectrum is $ \{ (2k+n)|\lambda|: k \in \N \}.$
Given a bounded function $ \varphi $ defined on the half line $
[0,\infty) $ we can define the operator $ \varphi(H(\lambda))$ by
spectral theorem. Taking $ M(\lambda) =  \varphi(H(\lambda))$ we
can consider the Fourier multiplier
$$ T_Mf(x,t) = (2\pi)^{-1/2} \int_{-\infty}^\infty e^{-i\lambda t} \varphi(H(\lambda))f^\lambda(x) d\lambda $$
where $ f \in L^p(\R,\R^n) = L^p(\R^{n+1}) $ and $ f^\lambda $
stands for the inverse Fourier transform of $ f $ in the $t $
variable. In this case the operator $ T_M $ can be interpreted as
a spectral multiplier for the Grushin operator $ -\Delta-|x|^2
\partial_t^2 $ on $ \R^{n+1}$ and such multipliers have been
studied in \cite{JST} and \cite{Ma}. It has been shown in
\cite{JST} that standard Hormander conditions on $ \varphi $ lead
to
R-Boundedness of the families $ \{ M(\lambda): \lambda \in \R \} $ and $ \{ \lambda M'(\lambda): \lambda \in \R \}.$\\

Another case where the R-boundedness of the multipliers can be proved  is given by Weyl multipliers. Using Theorem 1.3 we can easily prove the following result.

\begin{thm} For each $ \lambda \in \R $ let $ m(\lambda) \in B(L^2(\R^n)) $  and let $ \tilde{M}(\lambda)$ be the corresponding Weyl multiplier defined by $ W(\tilde{M}(\lambda)f^\lambda)= m(\lambda) W(f^\lambda).$ In our earlier notation, $ \tilde{M}(\lambda) = T_{m(\lambda)}.$
Assume that for each $ \lambda $ both  $ m(\lambda) $ and  $ \lambda
m'(\lambda) $ satisfy  the condition (1.1) for all
 $\alpha, \beta\in N^n, |\alpha|+|\beta|\leq 2n+2$. Then the operator valued Fourier multiplier defined by
$$ T_{\tilde{M}}f(z,t) = \int_{-\infty}^\infty e^{-i\lambda t} \tilde{M}(\lambda)f^\lambda(z) d\lambda $$
extends to be a  bounded operator on $ L^p (\R,L^p(\C^n))$ for any
$ 1< p < \infty.$
\end{thm}

Since we are considering $ X = L^p(\C^n) $ the R-boundedness of
the family $ \{\tilde{M}(\lambda):\lambda\in \R\} $ is equivalent
to the vector valued inequality for the sequence $
\tilde{M}(\lambda_j) $ for any choice of $ \lambda_j \in \R.$
According to a theorem of Rubio de Francia \cite{R} the vector
valued inequality will be a consequence of the weighted norm
inequality for $ \tilde{M}(\lambda)$:
$$ \int_{\C^n} |\tilde{M}(\lambda)f(z)|^2 w(z) dz \leq C~ \int_{\C^n} |f(z)|^2 w(z) dz $$ for all $ w \in A_2(\C^n) $ uniformly in $ \lambda $ which will follow from Theorem 1.3.\\

In  \cite{W} Weis has looked at the necessity of the conditions in
his theorem. He has proved  the following converse to his theorem.
Suppose the families $\{ M(\lambda):\lambda\in\R\} $ and $\{
\lambda M'(\lambda):\lambda\in \R\} $ are uniformly bounded on $
X.$ If the operator $ T_M $ is bounded on $ L^p(\R,X) $ then for
any $ a \neq 0 $ the family $ \{a2^nM(a2^n): n \in \Z \}$ is
R-bounded. However, it is not known if the R-boundedness of $
\{\lambda M'(\lambda):\lambda\in \R\} $ is necessary or not.
Our investigations on  Fourier multipliers on the Heisenberg group have led us to the following result.\\

\begin{thm}
Let $T_M$ be as in the theorem of Weis. The R-boundedness of
$\{\lambda M'(\lambda):\lambda\in \R\}$ is not necessary for the
boundedness of $T_M$ on $L^p(\R,X)$.
\end{thm}

As we have noted, when $ X = L^p(\C^n) $ the space $ L^p(\R, L^p(\C^n))$ can be
identified with $ L^p(\H^n) $ where $ \H^n$ is the Heisenberg
group. By considering the Riesz transforms on the Heisenberg group
which can be realised as operator valued Fourier multipliers, we
can prove the theorem stated above.\\

Coming back to Fourier multipliers on $ \H^n $ recall that the transforms $ T_m $ can be realised as

$$T_mf(z,t)=\int_{-\infty}^\infty e^{-i\lambda t} M(\lambda) f^\lambda(z)d\lambda$$
where $ W_\lambda(M(\lambda)f^\lambda) = m(\lambda)W_\lambda(f^\lambda).$
 It is therefore
natural to ask whether the R-boundedness of the families
$\{M(\lambda):\lambda\in\R^*\}$ and $\{\lambda
M'(\lambda):\lambda\in\R^*\}$ can be guaranteed by some conditions
on the multiplier $ m(\lambda) $ and its derivative $
m'(\lambda).$
When each $m(\lambda)$  is a Euclidean Fourier multiplier on $ L^p(\R^n) $ we have a simple result.\\

\begin{thm}
Let $\{ m(\lambda): \lambda\in \R\setminus \{0\}  \}$ be a family of
Euclidean Fourier multipliers on $L^p(\R^n), 1<p<\infty$ such that both the
families $\{m(\lambda): \lambda\in \R\setminus\{0\}\}$ and
$\{\lambda m'(\lambda): \lambda\in \R\setminus\{0\}\}$ are
R-bounded on $L^p(\R^n)$. Then the transform $T_m$ defined by
$\widehat{(T_m f)}(\lambda)=m(\lambda)\hat{f}(\lambda)$ on
$L^p\cap L^2(H^n)$ extends to $L^p(H^n)$ as a bounded linear
operator for $1<p<\infty$.
\end{thm}

But the story of general multipliers is quite different. We need
to find sufficient conditions on $m(\lambda)$ and $\lambda
m'(\lambda)$ so that the operator families $\{M(\lambda):
\lambda\in \R^*\}$ and $\{\lambda M'(\lambda):\lambda\in \R^*\}$
are R-bounded. As we have to deal with multipliers for $ W_\lambda $  as well as for $ W = W_1 $ it is conveneinet to use the following notation. Given a bounded linear operator $ S $ on $ L^2(\R^n) $ we use the notation $ T_S^\lambda $ to stand for the operator defined by
$ W_\lambda(T_S^\lambda g) = S W_\lambda(g).$ In this notation, $M(\lambda) = T^\lambda _{m(\lambda)}$ and we will use both notations  depending on the context. It can be shown that $ T_S^\lambda $ is conjugate to $ T_{\tilde{S}}^1 $ for some $ \tilde{S} $ which is related to $ S.$ We also need a family of non-commutative derivations depending on the parameter $ \lambda $.\\

 We define
$\delta_j(\lambda) m(\lambda)=|\lambda|^{-1/2}[m(\lambda), A_j(\lambda)] $ where $ A_j(\lambda)=\frac{\partial}{\partial
\xi_j}+|\lambda|\xi_j $ and  $
\bar{\delta_j}(\lambda)m(\lambda)=|\lambda|^{-1/2}[A^*_j(\lambda),
m(\lambda)],$ where
$A_j^*(\lambda)=-\frac{\partial}{\partial
\xi_j}+|\lambda|\xi_j.$
Considering the scaled Hermite
operator $ H(\lambda) $ which can be written as
 $$H(\lambda)=\frac{1}{2}\sum_{j=1}^n(A_j(\lambda)A_j^*(\lambda)+A_j^*(\lambda)A_j(\lambda))$$
 we define the dyadic spectral projections
$\chi_k(\lambda)=\sum_{2^{k-1}\leq 2j+n<2^k}P_j(\lambda).$ For the sake of brevity, let us
say that an operator $m$ satisfies the condition ($M_l(\lambda)$),
$M$ for Mauceri, if  $ m$ is of class $ C^l $ and satisfies the
following estimates:  For all  $\alpha, \beta\in N^n,
|\alpha|+|\beta|\leq l$
 $$\sup_{k\in \N^n} 2^{k(|\alpha|+\beta|-n)}||(\delta(\lambda)^\alpha \bar{\delta(\lambda)}^\beta m)\chi_k(\lambda)||^2_{HS}\leq C~$$
 with a constant $ C $ independent of $ \lambda $. We simply write
$(M_l)$ in place  of $(M_l(1)).$\\

If $ \delta_r f(z) = f(rz) $ stands for the dilation, then it  can be shown that
$\delta^{-1}_{\sqrt{|\lambda|}}T^\lambda_{m(\lambda)}\delta_{\sqrt{|\lambda|}}=T^1_{\tilde{m}(\lambda)}$
where
$\tilde{m}(\lambda)=\delta^{-1}_{\sqrt{|\lambda|}}m(\lambda)\delta_{\sqrt{|\lambda|}}$.
Equivalently, \be
M(\lambda)=T^\lambda_{m(\lambda)}=\delta_{\sqrt{|\lambda|}}T^1_{\tilde{m}(\lambda)}\delta^{-1}_{\sqrt{|\lambda|}}\ee
Note that $T^1_{\tilde{m}(\lambda)}$ is a Weyl multiplier:
$$W(T^1_{\tilde{m}(\lambda)}f)=\tilde{m}(\lambda)W(f), \;\;\;\;f\in L^2(\C^n).$$
According to a theorem of Rubio de Francia \cite{R} the
R-boundedness of the family $\{ M(\lambda) = T^\lambda_{m(\lambda)}:\lambda\in \R^*\}$
follows if we can prove  the weighted norm inequality
 $$\int_{\C^n}|T^\lambda_{m(\lambda)}f(z)|^2w(z)dz\leq C _w\int_{\C^n}|f(z)|^2w(z)dz$$
 for all $w\in A_2(\C^n)$ uniformly in $\lambda$. In view of (1.3)  it
 is enough to prove this inequality for $T^1_{\tilde{m}(\lambda)}$, $\lambda\in
 \R^*$. This leads to the following result.\\

 \begin{thm}
 For every $ \lambda \in \R^* $ let $m(\lambda) \in B(L^2(\R^n))$ satisfy the condition $
 ( M_{2n+2}(\lambda)).$
 Then for every $ 1 < p < \infty,$ the Weyl multiplier $ M(\lambda) =T^\lambda_{m(\lambda)}$ is R-bounded on $L^p(\C^n).$
\end{thm}

This takes care of the R-boundednesss of the family $ \{ M(\lambda): \lambda \in \R^* \}.$ The R-boundedness of the family $\{\lambda
M'(\lambda):\lambda \in \R^*\}$ turns out to be even more complicated.
Since we have
$$M(\lambda)=\delta_{\sqrt{|\lambda|}}T^1_{\tilde{m}(\lambda)} \delta^{-1}_{\sqrt{|\lambda|}}$$
 the derivative of $M(\lambda)$ involves several terms. We will show
that
$$2\lambda \frac{d}{d\lambda}M(\lambda)=[B,M(\lambda)]+T^\lambda_{[x.\nabla, m(\lambda)]}+T^\lambda_{2\lambda \frac{d}{d\lambda}m(\lambda)}$$
where $\nabla=(\frac{\partial}{\partial
x_1},...\frac{\partial}{\partial x_n})$ and $B=\sum_{j=1}^n
(z_j\frac{\partial}{\partial
z_j}+\bar{z_j}\frac{\partial}{\partial \bar{z_j}}).$ The second
and third terms are easy to handle whereas the first term is not.
It turns out that $[B, M(\lambda)]$ is not even a Weyl multiplier
and hence not accessible to our methods.\\

The Riesz transforms $R_j$ on the Heisenberg group $H^n$ are
defined via the multipliers $A_j(\lambda)H(\lambda)^{-1/2}$ and it is well
known that they are bounded on $L^p(H^n)$. In this case it turns
out $\{[B,M(\lambda)]:\lambda\in \R^*\}$ is not R-bounded. As a
consequence of this we obtain Theorem 1.6. As we discussed
earlier, because of the behavior of $ [B,M(\lambda)]$ we cannot
get any sufficient condition for $L^p$-boundedness of the operator
$T_m$ in terms of the condition
(M). However we have the following result.\\

\begin{thm}
Let the two families of operators $\{m(\lambda): \lambda\in
\R^*\}$, $\{\lambda m'(\lambda): \lambda\in \R^*\}$ satisfy the
conditions ($M_{2n+3}(\lambda)$) and ($M_{2n+2}(\lambda)$)
respectively. Let  $\mathcal{L}$ stand for the sublaplacian on $H^n$.Then the multiplier transform $ T_m $ on $ \H^n
$ satisfies
$$||T_{m} f||_p\leq C~ ||\mathcal{L}^{1/2}f||_p,\; \;\; 1<p<\infty.$$

\end{thm}

We also have the following result. Let $T(n)\subset U(n)$ be the torus which acts on $\C^n$ by
$\rho(\sigma)f(z)=f(e^{i\theta_1} z_1, ...,e^{i\theta_n} z_n)$ if
$\sigma$ is the diagonal matrix with entries $e^{i\theta_1}, ...,
e^{i\theta_n}$. Then
$$Rf(z)=\int_{T(n)} \rho(\sigma)f(z) d\sigma$$
is polyradial and it can be easily checked that $||Rf||_p\leq ||f||_p$.\\

\begin{thm}
Let the two families of operators $\{m(\lambda): \lambda\in
\R^*\}$, and $\{\lambda m'(\lambda): \lambda\in \R^*\}$ satisfies
the condition ($M_{2n+3}(\lambda)$) and ($M_{2n+2}(\lambda)$)
respectively. Then for the multiplier transform $ T_m $ on $ \H^n
$ we have
$$||R\circ  T_m \circ Rf||_p\leq C~ ||f||_p,\;\;\;1<p<\infty.$$
\end{thm}

We conclude the introduction with the following remarks. It is possible to improve slightly the results of Theorems 1.2 and 1.3. In a recent article \cite{BT} we have studied the $ L^p $  boundedness  of Hermite pseudo-multipliers. In that connection we have  introduced modified Mauceri conditions. Following the ideas  and techniques used in that paper we can prove theorem 1.2  for multipliers of class $ C^{n+1}.$ Also in Theorem 1.3 we can reduce the number of derivatives from $ 2n+2 $ to $2n+1.$ The plan of the paper is as follows. In the next section we set up
notation, recall results from the theory of weyl transforms and
prove some preliminary lemmas needed later. In Section 3
we take up the problem of estimating certain kernels associated to Weyl multipliers. In Section 4  we prove our main results.\\

\section[Preliminaries]
{Preliminaries} We consider the Heisenberg group $H^n=\C^n\times
\R$ equipped with the group law
$$(z,t)(w,s)=(z+w, t+s+\frac{1}{2}\Im(z.\bar{w})).$$
This is a step two nilpotent Lie group and the Haar measure on
$H^n$ is simply the Lebesgue measure $dzdt$ on $\C^n\times \R$. In
order to define the Fourier transform on $H^n$ we need to recall
certain families of irreducible unitary representations of $H^n$.\\

For each $\lambda\in \R^*= \R\setminus \{0\}$ and $(z, t)\in H^n$
consider the operator $\pi_\lambda(z,t)$ defined on $L^2(\R^n)$ by
$$\pi_\lambda(z,t)\varphi(\xi)=e^{i\lambda t} e^{i\lambda(x.\xi+\frac{1}{2}x.y)}\varphi(\xi+y)$$
where $\varphi\in L^2(\R^n)$. It can be shown that each
$\pi_\lambda$ is an irreducible unitary representation of $H^n$.
Moreover, by a theorem of Stone-von Neumann any irreducible
unitary representation of $H^n$ which is non-trivial at the center
of $H^n$ is unitarily equivalent to $\pi_\lambda$ for a unique
$\lambda\in \R^*$. Apart from $\pi_\lambda$, there is another
family of one dimensional irreducible unitary representations. As
they do not play any role in the Plancherel theorem we do not
consider them. See Folland \cite{F} and Thangavelu \cite{T1} for
more on these
representations.\\

Given $f\in L^1(H^n)$ we can define the operator
$\hat{f}(\lambda)=\pi_\lambda(f)$ by
$$\hat{f}(\lambda)=\int_{\C^n} f(z,t)\pi_\lambda(z,t)dzdt.$$
The operator valued function $\lambda\rightarrow \hat{f}(\lambda)$
is called the (group) Fourier transform of $f$ on $H^n$. If we let
$f^\lambda$ stand for the inverse Fourier transform of $f$ in the
t-variable, i.e.
$$f^\lambda (z)=\int_{-\infty}^\infty e^{i\lambda t} f(z, t) dz dt$$
then we have
$$\hat{f}(\lambda)=\int_{\C^n}f^\lambda(z) \pi_\lambda (z,0)dz.$$
When $f\in L^1\cap L^2(H^n)$, it can be shown that
$\hat{f}(\lambda)$ is a Hilbert-Schmidt operator and
$$||\hat{f}(\lambda)||^2_{HS}=(2\pi)^n|\lambda|^{-n}\int_{\C^n}|f^\lambda (z)|^2dz.$$
In view of this Plancherel theorem for the Fourier transform takes
the form
$$||f||^2_2= \int_{-\infty}^{\infty}||\hat{f}(\lambda)||^2_{HS} d\mu(z)$$
where $d\mu(\lambda)=(2\pi)^{-n-1} |\lambda|^nd\lambda$ is the
Plancherel measure. We also have the inversion formula
$$f(z,t)=\int_{-\infty}^{\infty}tr(\pi_\lambda(z,t)^* \hat{f}(\lambda)) d\mu(\lambda)$$
for suitable functions.\\

Given a family of bounded linear operators $m(\lambda), \lambda\in
\R^*$ we can define $T_m$ by $ \widehat{(T_mf)}(\lambda) =
m(\lambda)\hat{f}(\lambda)$ which can also be written in the form
$$ T_mf(z,t) = (2\pi)^{-1} \int_{-\infty}^\infty e^{-i\lambda t} T^\lambda_{m(\lambda)}f^\lambda(z)  d\lambda.$$
The study of $ T^\lambda_{m(\lambda)} $ can be reduced to the study of Weyl multipliers using the following lemma.\\
\begin{lem}
For each $\lambda>0$ we have $\delta^{-1}_{\sqrt{\lambda}}
T^\lambda_{m(\lambda)}\delta_{\sqrt{\lambda}}=T^1_{\tilde{m}(\lambda)}$
where $\tilde{m}(\lambda)=\delta^{-1}_{\sqrt{\lambda}}
m(\lambda)\delta_{\sqrt{\lambda}}$.
\end{lem}
\begin{proof}For $\lambda>0$ an easy calculation shows that
$$\pi_\lambda(z,t)=e^{i\lambda t} \delta_{\sqrt{\lambda}}\pi_1 (\sqrt{\lambda}z,0) \delta^{-1}_{\sqrt{\lambda}}.$$
Since $W_\lambda(f)$ is defined in terms of $\pi_\lambda(z,0)$ it
follows that
$$W_\lambda(\delta_{\sqrt{\lambda}}f)=\lambda^{-n} \delta_{\sqrt{\lambda}} W(f) \delta^{-1}_{\sqrt{\lambda}}$$
for any $f\in L^2(\C^n)$. Applying this identity to
$\delta^{-1}_{\sqrt{\lambda}}T^\lambda_{m(\lambda)}\delta_{\sqrt{\lambda}}f$
we see that
$$m(\lambda)W_\lambda(\delta_{\sqrt{\lambda}}f)=
\lambda^{-n}\delta_{\sqrt{\lambda}}W(\delta^{-1}_{\sqrt{\lambda}}T^\lambda_{m(\lambda)}\delta_{\sqrt{\lambda}}f)\delta^{-1}_{\sqrt{\lambda}}.$$
This immediately gives
$$W(\delta^{-1}_{\sqrt{\lambda}}T^\lambda_{m(\lambda)}\delta_{\sqrt{\lambda}}f)=\delta^{-1}m(\lambda)\delta_{\sqrt{\lambda}}W(f)$$
as desired.
\end{proof}

From the above lemma it is clear that, though
$T^\lambda_{m(\lambda)}$ is not a Weyl multiplier,
$\delta^{-1}_{\sqrt{\lambda}}T^\lambda_{m(\lambda)}\delta_{\sqrt{\lambda}}$
is. The noncommutative derivations $\delta_j$ and $\bar{\delta}_j$
acting on $\tilde{m}(\lambda)$ can be converted into certain
derivations acting on $m(\lambda)$ itself.\\

Recall that for an operator $m$ on $L^2(\R^n)$ we have defined
$\delta_j m=[m, A_j]$ and $\bar{\delta_j}m=[A_j^*, m]$ where
$A_j=\frac{\partial}{\partial \xi_j}+\xi_j$ and
$A_j^*=-\frac{\partial}{\partial \xi_j}+\xi_j. $ Under the
dilation $\delta_{\sqrt{\lambda}}$ we have
$$\delta_{\sqrt{\lambda}}A_j \delta^{-1}_{\sqrt{\lambda}}=\lambda^{-1/2}(\frac{\partial}{\partial \xi_j}+\lambda \xi_j)=
\lambda^{-1/2}A_j(\lambda)$$ and
$$\delta_{\sqrt{\lambda}}A_j^* \delta^{-1}_{\sqrt{\lambda}}=\lambda^{-1/2}(-\frac{\partial}{\partial \xi_j}+\lambda \xi_j)=
\lambda^{-1/2}A_j^*(\lambda).$$ In view of these relations,
$\delta_j
\tilde{m}(\lambda)=\lambda^{-1/2}\delta^{-1}_{\sqrt{\lambda}}[m(\lambda),
A_j(\lambda)]\delta_{\sqrt{\lambda}} $ and
$\bar{\delta_j}\tilde{m}(\lambda)=\lambda^{-1/2}\delta^{-1}_{\sqrt{\lambda}}[A_j^*(\lambda),m(\lambda)]\delta_{\sqrt{\lambda}}.$\\

\begin{lem}
For $\lambda> 0$ we have $\delta_j
\tilde{m}(\lambda)=\lambda^{-1/2}\delta^{-1}_{\sqrt{\lambda}}\delta_j(\lambda)m(\lambda)\delta_{\sqrt{\lambda}}$
and
$\bar{\delta_j}\tilde{m}(\lambda)=\lambda^{-1/2}\delta^{-1}_{\sqrt{\lambda}}\bar{\delta}_j(\lambda)m(\lambda)\delta_{\sqrt{\lambda}}$
where $\delta_j(\lambda) m(\lambda)=\lambda^{-1/2}[m(\lambda),
A_j(\lambda)]$ and $\bar{\delta_j}(\lambda)m(\lambda)=
\lambda^{-1/2}[A^*_j(\lambda),m(\lambda)]$.

\end{lem}
On the Heisenberg group we have the left invariant vector fields
$$ T = \frac{\partial}{\partial t}, X_j=\frac{\partial}{\partial x_j} +\frac{1}{2} y_j \frac{\partial}{\partial
t},\; Y_j=\frac{\partial}{\partial y_j} +\frac{1}{2} x_j
\frac{\partial}{\partial t}$$ which give rise to a family of
operators $Z_j(\lambda)$ and $\bar{Z_j}(\lambda)$ as follows:
$\frac{1}{2} (X_j-iY_j)(e^{i\lambda t} f(z))= e^{i\lambda t}
Z_j(\lambda)f(z)$ and $\frac{1}{2} (X_j+iY_j)(e^{i\lambda t}
f(z))= e^{i\lambda t} \bar{Z_j}(\lambda)f(z).$ Explicitly, we have
$Z_j(\lambda)=\frac{\partial}{\partial z_j} -\frac{\lambda}{4}
\bar{z_j}$ and $\bar{Z}_j(\lambda)=\frac{\partial}{\partial
\bar{z_j}} +\frac{\lambda}{4} z_j.$ We also have their right
invariant counterparts $X_j^R=\frac{\partial}{\partial x_j}
+\frac{1}{2} y_j \frac{\partial}{\partial t}$ and $
Y^R_j=\frac{\partial}{\partial y_j} +\frac{1}{2} x_j
\frac{\partial}{\partial t}$. These gives us the vector fields
$Z^R_j(\lambda)=\frac{\partial}{\partial z_j} +\frac{\lambda}{4}
\bar{z_j}$ and $\bar{Z}^R_j(\lambda)=\frac{\partial}{\partial
\bar{z_j}} -\frac{\lambda}{4} z_j.$
 We record the following properties for later use.\\

 \begin{lem}
 For any $\lambda>0$, $f\in L^2(\C^n)$ we have:
\begin{enumerate}
\item $W_\lambda(Z_j(\lambda)f)=i W_\lambda(f)A_j^*(\lambda)$, $
W_\lambda(\bar{Z_j}(\lambda)f)=iW_\lambda(f)A_j(\lambda),$ \item
$W_\lambda(Z^R_j(\lambda)f)=i A_j^*(\lambda)W_\lambda(f)$,
$W_\lambda(\bar{Z^R_j}(\lambda)f)=i A_j(\lambda)W_\lambda(f),$
\item $\lambda W_\lambda(z_jf)=2i [W_\lambda (f),
A_j^*(\lambda)]$, $\lambda W_\lambda(\bar{z}_jf)=2i
[A_j(\lambda),W_\lambda (f)].$
\end{enumerate}
 \end{lem}
 The role of Laplacian $\Delta$ for $H^n$ is played by the
 sublaplacian $\mathcal{L}$ defined by
 $$\mathcal{L}=-\sum^n_{j=1}(X_j^2+Y_j^2)=-\frac{1}{2} \sum^n_{j=1} (X_j-Y_j)(X_j+Y_j)+(X_j+Y_j)(X_j-Y_j).$$
 The operator $L_\lambda$ defined by the condition $\mathcal{L}(e^{i\lambda
 t} f(z)=e^{i\lambda t} L_\lambda(f(z))$ is called the special
 Hermite operator which can be written as
 $$L_\lambda=-2 \sum_{j=1}^n Z_j(\lambda) \bar{Z_j}(\lambda)+\bar{Z}_j(\lambda)Z_j(\lambda).$$
 In view of Lemma 2.3 we have
 $$W_\lambda(L_\lambda f)=W_\lambda(f)\sum_{j=1}^n(A_j^*(\lambda)A_j(\lambda)+A_j(\lambda)A_j^*(\lambda))=W_\lambda(f) H(\lambda)$$
 for functions on $\C^n$. This leads to the equation
 $\widehat{(\mathcal{L}f)}(\lambda)=\hat{f}(\lambda)H(\lambda)$ for functions
 on $H^n$. If $\mathcal{L}^R$ stand for the right invariant
 sublaplacian then we have $\widehat{(\mathcal{L}^R f)}(\lambda)=H(\lambda)\hat{f}(\lambda). $\\

 In order to study the R-boundedness of the family $\lambda
 M'(\lambda)=\lambda \frac{d}{d\lambda}T^\lambda_{m(\lambda)}$ we
 need to get a usable expression for the derivative. We introduce
 the following operators. Let $\nabla$ stand for the gradient on
 $\R^n$,
 $\nabla=(\frac{\partial}{\partial\xi_1},\frac{\partial}{\partial\xi_2},...\frac{\partial}{\partial\xi_n})$
 and let $\xi \cdot\nabla=\sum_{j=1}^n \xi_j \frac{\partial}{\partial\xi_j}$. On $\C^n$ we let $$B=\sum^n_{j=1}(z_j \frac{\partial}{\partial
 z_j}+\bar{z_j}
 \frac{\partial}{\partial\bar{z_j}})=\sum_{j=1}^N(x_j\frac{\partial}{\partial
 x_j}+y_j \frac{\partial}{\partial y_j}).$$
 The dilation operator $\delta_{\sqrt{\lambda}}, \lambda>0$ on
 $\R^n$ can be expressed as
 $$\delta_{\sqrt{\lambda}}\varphi(\xi)=e^{\frac{1}{2} \xi \cdot \nabla (\log \lambda)} \varphi(\xi)$$
 and the same on $\C^n$ can be written as
 $$\delta_{\sqrt{\lambda}} f(z)=e^{\frac{1}{2}B(\log \lambda)}f(z).$$
 Using these expressions we can easily prove the following.\\
 \begin{lem}
 For $\lambda>0$ we have
 $$2 \lambda \frac{d}{d\lambda} T^\lambda_{m(\lambda)}=
 [B, T^\lambda_{m(\lambda)}]+T^\lambda_{[m(\lambda), \xi \cdot \nabla]}+T^\lambda_{2\lambda \frac{d}{d\lambda}m(\lambda)}.$$
 \end{lem}
 \begin{proof}
 Differentiating the equation
 $T^\lambda_{m(\lambda)}= \delta_{\sqrt{\lambda}}T^1_{\tilde{m}(\lambda)}\delta^{-1}_{\sqrt{\lambda}}$
 we get
 $$2\lambda \frac{d}{d\lambda}T^\lambda_{m(\lambda)}=\delta_{\sqrt{\lambda}}[B, T^1_{\tilde{m}(\lambda)}]\delta_{\sqrt{\lambda}}^{-1}+
 \delta_{\sqrt{\lambda}}2 \lambda \frac{d}{d\lambda}T^1_{\tilde{m}(\lambda)}\delta_{\sqrt{\lambda}}^{-1}.$$
 Since $\tilde{m}(\lambda)=\delta_{\sqrt{\lambda}}^{-1} m(\lambda)
 \delta_{\sqrt{\lambda}},$ by differentiating the equation
 $$W(T^1_{\tilde{m}(\lambda)}f)=\delta_{\sqrt{\lambda}}^{-1} m(\lambda)\delta_{\sqrt{\lambda}} W(f)$$
 and using the expression for $\delta_{\sqrt{\lambda}}$ we get
 $$2 \lambda \frac{d}{d\lambda} \tilde{m}(\lambda)=\delta_{\sqrt{\lambda}}^{-1} 2\lambda \frac{d}{d\lambda}m(\lambda) \delta_{\sqrt{\lambda}}+
 \delta_{\sqrt{\lambda}}^{-1}[m(\lambda), \xi \cdot\nabla]\delta_{\sqrt{\lambda}}.$$
 As $B$ and $\xi \cdot \nabla$ commute with $\delta_{\sqrt{\lambda}}$ we
 get
 $$\delta_{\sqrt{\lambda}}[B, T^1_{\tilde{m}(\lambda)}] \delta_{\sqrt{\lambda}}^{-1}= [B, T^\lambda_{m(\lambda)}]$$
 and
 $$\delta_{\sqrt{\lambda}}^{-1}[m(\lambda), \xi \cdot \nabla] \delta_{\sqrt{\lambda}}=[\tilde{m}(\lambda), \xi\cdot\nabla].$$
 This completes the proof.

 \end{proof}
\section[Weighted norm estimates for Weyl multipliers]
{Weighted norm estimates for Weyl multipliers}
\setcounter{equation}{0}
 \setcounter{equation}{0}
 In this section our aim is to show that when the multiplier $m$
 satisfies condition $(M_{2(n+1)})$ the Weyl multiplier $T_m$ is
 bounded on $L^p(\C^n, w)$ for all $w\in A_p(\C^n), 1<p<\infty$.
 Any such Weyl multiplier is a twisted convolution operator: $T_m
 f=k\times f$ for a distribution $k$ on $\C^n$. We will show that
 conditions on $m$ can be translated into estimates on the kernel
 $k$ which will then be used to prove the weighted norm
 inequality. We begin with the following result.
 \begin{thm}
 Consider the operator $Tf=k\times f$ where $k\in L^2(\C^n)$
 satisfies the estimate
 \be
|k(z-u)e^{-\frac{i}{2} \Im (z.\bar{u})}-k(z)|\leq C~
 \frac{|u|^\delta}{|z|^{2n+\delta}}
 \ee
 for some $\delta>0$ and all $|z|>2|u|$. Assume that $T$ is
 bounded on $L^p(\C^n)$ and $Tf\in L^p(\C^n, w)$ for a dense class
 of functions in $L^p(\C^n, w)$ for all $w\in A_p(\C^n),
 1<p<\infty$. Then
 $$\int_{\C^n}|Tf|^p w(z) dz \leq C~ \int_{\C^n} |f(z)|^pw(z)dz. $$

 \end{thm}
 The proof of this theorem uses standard arguments. It is
 well-known that Calderon-Zygmund operators are bounded on
 $L^p(\R^n, w), w\in A_p(\R^n), 1<p<\infty$, see e.g. Theorem 7.11
 in \cite{D}. Our operator $T $ is sort of an oscillatory singular integral
 operator and hence the arguments used in proving Theorem 7.11 in
 \cite{D} can be suitably modified to prove Theorem 3.1.\\

 The sharp maximal function $M^\sharp$ used in the literature needs to be
 modified. We define the twisted sharp maximal function
 $\tilde{M}^\sharp$ by
 $$\tilde{M}^\sharp f(v)=\sup_{v\in Q}\frac{1}{|Q|}\int_Q |f(z)e^{-\frac{i}{2}\Im(z\cdot \bar{u})}-\tilde{f_Q}|dz$$
 where $f$ is a locally integrable function, $Q$ is a cube, $u$
 its center and
 $$\tilde{f}_Q=\frac{1}{|Q|}\int_Q f(z)e^{-\frac{i}{2}\Im(z\cdot \bar{u})}dz.$$
 For any $s>1$ let $M_s f=(M|f|^s)^{\frac{1}{s}}$, $M$ being the
 Hardy-Littlewood maximal function.
 \begin{lem}
 Let $k\in L^2(\C^n)$ satisfies the condition (3.1) in the above
 theorem. If $Tf=k\times f$ is bounded on $L^p(\C^n), 1<p<\infty$
 then
 $$\tilde{M}^\sharp (Tf)(v)\leq C~ M_s f(v)$$
 for any $1<s<\infty$.
 \end{lem}
 The proof of this lemma is similar to that of Lemma 7.9 in \cite{D}.
 All we have to do is use $\tilde{f}_Q$ in place of $f_Q$.

Let $M_d$ stand for the dyadic maximal function  (see Section 5, Chapter 2 in \cite{D}).\\
 \begin{lem}
 Let $w\in A_p(\C^n), 1\leq p_0\leq p<\infty$. Then
 $$\int_{\C^n}|M_df(z)|^pw(z)dz\leq C~\int_{\C^n}|\tilde{M}^\sharp f(z)|^p w(z) dz$$
 whenever $M_d f\in L^{p_0}(\C^n,w)$.
 \end{lem}

 The proof of this lemma depends on good-$\lambda$ inequality: for
 some $\delta>0$
 $$w(\{z\in \C^n: M_df(z)>2\lambda, \tilde{M}^\sharp f(z)\leq \gamma \lambda\})$$
$$ \leq C~ \gamma^\delta w(\{z\in \C^n:M_df(z)>\lambda\}).$$
 Once we have this inequality, the lemma can be proved by
 expressing the $L^p$ norm of $M_d f$ in terms of its distribution
 function. See the proof of Lemma 6.9 in \cite{D}.\\

 The good-$\lambda$ inequality with $M^\sharp$ in place of
 $\tilde{M}^\sharp$ has been proved in \cite{D} (see Lemma 6.10 and Lemma
 7.10). The same proof goes through with slight modifications on
 account of the 'twist'. We leave the details to the reader.\\

 It is now easy to prove Theorem 3.1. By the hypothesis, there
 is a dense class $\mathcal{D}\subset L^p(\C^n, w)$ such that $Tf\in
 L^p(\C^n, w)$ for $f\in \mathcal{D}$. As $w\in A_p(\C^n)$, there
 exists $s, 1<s<p$ such that $w\in A_{p/s}(\C^n)$. For
 $f\in \mathcal{D}$, $|Tf(z)|\leq M_d (Tf)(z)$ a.e and hence
 $$\int_{\C^n}|Tf(z)|^p w(z)dz\leq \int_{\C^n} (M_d(Tf)(z))^pw(z)dz.$$
 By Lemma 3.2, 3.3 and the boundedness of $M_s$ we get
 $$\int_{\C^n} |Tf(z)|^p w(z)dz\leq C~ \int_{\C^n}|f(z)|^p w(z)dz.$$
 This completes the proof of Theorem 3.1.\\

 The existence of a dense class of functions $\mathcal{D}\subset
 L^p(\C^n, w)$ such that $Tf\in L^p(\C^n, w), f\in \mathcal{D}$ is
 guaranteed once we assume another estimate on the kernel $k$.
 Indeed, under the assumption
 \be
 |k(z)|\leq C~ |z|^{-2n-\theta}\;\;\;\;\;\text{for some } \theta >0
\ee we can show that $Tf\in L^p(\C^n, w)$ whenever $f\in
C^\infty_0(\C^n)$ (which is dense in $L^p(\C^n, w)$). This has
been proved for Calderon-Zygmund operators in Theorem 3.11 of
\cite{D}. As it only uses the size estimate our assertion is
proved. We can now restate Theorem 3.1 in the following form.
\begin{thm}
Consider the operator $Tf=k\times f $ where $k\in L^2 (\C^n)$
satisfies (3.1) and (3.2). If $T$ is bounded on $L^p(\C^n)$,
$1<p<\infty$ then it is also bounded on $L^p(\C^n,w)$, for all
$w\in A_p(\C^n), 1<p<\infty$.
\end{thm}
Thus in order to prove weighted norm inequalities for $T_m$ we
only need to prove estimates (3.1) and (3.2) for the kernel $k$ of
$T_m$.
\begin{thm}
Let $m\in B(L^2(\R^n))$ be of class $C^{2n+2}$ and satisfy
Mauceri's condition $(M_{2n+2})$. Then
\begin{enumerate}
\item $|k(z)|\leq C~ |z|^{-2n-\theta}$ for some $\theta\geq 0,$
\item $|k(z-u)e^{-\frac{i}{2}\Im(z\cdot \bar{u})}-k(z)|\leq
\frac{|u|^\delta}{|z|^{2n+\delta}}$
\end{enumerate}
for some $\delta > 0$ and for all $|z|> 2|u| $ where $k(z)$ is the
kernel of the operator $M=T_m$.
\end{thm}

In order to prove the theorem, let $t_j=2^{-j}, j=1,2,...$ and
consider
$$ S_j = \sum_{k=0}^\infty (e^{-2kt_j}-e^{-2k t_{j+1}})P_k = e^{n t_j} e^{-t_j H}-e^{nt_{j+1}}e^{-t_{j+1}H}.$$
Then it follows that $\sum_{j=1}^N S_j=e^{n t_1} e^{-t_1
H}-e^{nt_{N+1}}e^{-t_{N+1}H}$ and taking limit as $N\rightarrow
\infty$ we get $I= e^{n t_1} e^{-t_1 H}-\sum_{j=1}^\infty S_j.$
Using this we decompose our operator $m$ as
$$m=m_0-\sum_{j=1}^\infty m_j, \;m_j=mS_j,\; m_0=me^{nt_1}e^{-t_1H}.$$
Let $k_j$ stand for the kernel of $m_j, j=0,1,2,...$\\

\begin{prop}
For each $j=0,1,2,...$ we have
\begin{enumerate}
\item $|k_j(z)|\leq C~~ t_{j+1}^{1/4}|z|^{-2n-1/2}$ \item
$|k_j(z-u)e^{-\frac{i}{2}\Im(z \cdot \bar{u})}-k_j(z)|\leq
C~\frac{|u|^{1/2}}{|z|^{2n+1/2}}
\min(\frac{t_j^{1/4}}{|u|^{1/2}},\frac{|u|^{1/2}}{t_{j+1}^{1/4}})$
\end{enumerate}
for all $|z|> 2|u|$.
\end{prop}

Theorem 3.5 follows immediately once we prove this proposition.
 Indeed,
 $$|k(z)e^{-\frac{i}{2}\Im(z \cdot \bar{u})}-k(z)|\leq  C \frac{|u|^{1/2}}{|z|^{2n+1/2}}
\sum_{j=0}^\infty
\min(\frac{t_j^{1/4}}{|u|^{1/2}},\frac{|u|^{1/2}}{t_{j+1}^{1/4}})$$
and splitting the sum into two parts we see that
$$\frac{|u|^{1/2}}{|z|^{2n+1/2}}
\sum_{t_{j+1}\leq|u|^2} t_{j+1}^{1/4} |u|^{-1/2}\leq C~
\frac{|u|^{1/2}}{|z|^{2n+1/2}}$$ and also
$$\frac{|u|^{1/2}}{|z|^{2n+1/2}}
\sum_{t_{j+1}>|u|^2} t_{j+1}^{-1/4} |u|^{1/2}\leq C
\frac{|u|^{1/2}}{|z|^{2n+1/2}}.$$ Thus we only
need to prove Proposition 3.6.\\

Coming to the proof of Proposition 3.6 we claim that for all $z\in
\C^n$
 \be |z|^l |k_j(z)|\leq C~ t_{j+1}^{l/2-n}\ee
whenever $l\leq 2n+1$. In order to estimate $ |z|^l k_j(z) $ it is
enough to estimate $z^\alpha \bar{z}^\beta k_j(z)$ where
$|\alpha|+|\beta|=l$. Under the Weyl transform $ z^\alpha
\bar{z}^\beta k_j(z)$ goes into $\bar{\delta}^\alpha \delta^\beta
(mS_j)$ which by Leibniz formula for the derivations
$\bar{\delta}^\alpha$ and $\delta^\beta$ is a sum of terms of the
form
$$(\bar{\delta}^\mu \delta^\nu m)(\bar{\delta}^\gamma \delta^\rho S_j),\;\;\;|\mu|+|\nu|+|\gamma|+|\rho|=l.$$
We decompose each of these operators as
$$\sum_{N=0}^\infty (\bar{\delta}^\mu \delta^\nu m)\chi_N. \chi_N(\bar{\delta}^\gamma \delta^\rho S_j).$$
Since $|f\times g(z)|\leq ||f||_2||g||_2$, the $L^\infty$ norm of
the kernel of $(\bar{\delta}^\mu \delta^\nu m)(\bar{\delta}^\gamma
\delta^\rho S_j)$is bounded by
$$\sum_{N=0}^\infty ||(\bar{\delta}^\mu \delta^\nu m)\chi_N||_{HS} ||\chi_N(\bar{\delta}^\gamma \delta^\rho S_j)||_{HS}.$$
We now make use of the following lemma which is essentially Lemma
4.4  proved in Mauceri \cite{M}.
\begin{lem} For every $ \gamma $ and $ \rho $ we have the estimate
$$||\chi_N \bar{\delta}^\gamma \delta^\rho S_j||^2_{HS}\leq C~
t^2_{j+1} 2^{N(n+2-|\gamma|-|\rho|)}f^2_{\gamma, \rho}(2^N
t_{j+1})$$ where $f_{\gamma,\rho}$ is a rapidly decreasing
function.
\end{lem}

In view of this lemma, the kernel of $(\bar{\delta}^\mu \delta_\nu
m)(\bar{\delta}^\gamma \delta^\rho S_j)$ is bounded by constant
times
$$
t_{j+1}\sum_{N=0}^\infty
2^{\frac{N}{2}(n-|\mu|-|\nu|)}2^{\frac{N}{2}(n+2-|\gamma|-|\rho|)}f_{\gamma,
\rho}(2^N t_{j+1})$$ $$=t_{j+1}\sum_{N=0}^\infty
2^{\frac{N}{2}(2n+2-l)}f_{\gamma,\rho}(2^Nt_{j+1})$$ $$\leq
C~~t_{j+1}^{-n+\frac{l}{2}}\sum_{N=0}^\infty (2^N
t_{j+1})^{(n+1-\frac{l}{2})}f_{\gamma,\rho}(2^N t_{j+1}).$$
 Since $f_{\gamma,\rho}(2^N t_{j+1})$ has exponential decay $\sum_{N=0}^\infty (2^N)^\alpha f_{\gamma,\rho}(2^N t_{j+1})$
 converges leading to the estimate $ C~ t_{j+1}^{-\alpha}$ for each $\alpha>0$. Hence the above series  can be
 estimated by
 $C_{\gamma,\rho}t_{j+1}^{-n+\frac{l}{2}}$. As this is true for every $\mu, \nu,\gamma$ and $\rho$
 satisfying $|\mu|+|\nu|+|\gamma|+|\rho|=|\alpha|+|\beta|=l$ we get
 the estimate
 \be
 |z^\alpha \bar{z}^\beta k_j(z)|\leq C_{\alpha,\beta}
 t_{j+1}^{-n+\frac{l}{2}},\;\;\; |\alpha|+|\beta|=l
 \ee
 which leads to $|z|^l|k_j(z)|\leq C_lt_{j+1}^{-n+\frac{l}{2}}$.\\

 When $l=2n$ we get $|k_j(z)|\leq C~ |z|^{-2n}$ and when $l=2n+1$ we
get $|k_j(z)| \leq C~ t_{j+1}^{1/2}|z|^{-2n-1}$ combining these
two estimates we obtain\be |k_j(z)|\leq C~
t_{j+1}^{1/4}|z|^{-2n-1/2}.\ee Again if we take $l=2n+2$ the above
series is estimated by
$$t_{j+1}\sum_{N=0}^\infty f_{\gamma,\rho}(2^N
t_{j+1})\leq \sum_{N=0}^\infty (2^Nt_{j+1})f_{\gamma,\rho}(2^N
t_{j+1})\leq C$$ Hence we also obtain the following inequality
$$|k_j(z)|\leq C~ |z|^{-2n-2}.$$ Thus we have proved (1) of Proposition 3.6. In order to prove
(2) we need to estimate the gradient of $k_j$ for which we proceed
as
follows.\\

Since $\frac{\partial}{\partial z_r}k_j=Z_rk_j-\frac{1}{4}
\bar{z_r} k_j$ and $W(Z_r k_j)=i mS_j A_r^*$, in order to estimate
$|z|^l \frac{\partial}{\partial z_r}k_j$ we have to estimate
$$\sum_{N=0}^\infty ||(\bar{\delta}^\mu \delta^\nu m)\chi_N||_{HS}||\chi_N \bar{\delta}^\gamma \delta^\rho (S_j A_r^*)||_{HS}$$
where $|\mu|+|\nu|+|\gamma|+|\rho|=l$. Since  $\bar{\delta_r}
A_r^*=0$ and $\delta_r A_r^*=2I$ it is enough to estimate
$$\sum_{N=0}^\infty ||(\bar{\delta}^\mu \delta^\nu m)\chi_N||_{HS}||\chi_N(\bar{\delta}^\gamma \delta^\rho S_j) A_r^*||_{HS}$$
We use the Hermite basis $\Phi_\alpha, \alpha\in \N^n$ to
calculate the Hilbert-Schmidt norm. Since $A_r^* \Phi_\alpha
=(2|\alpha|+2+n)^{1/2} \Phi_{\alpha +e_r}$ it follows that
$$||\chi_N(\bar{\delta}^\gamma \delta^\rho S_j) A_r^*||_{HS}^2\leq C~ t_{j+1}^2 2^{N(n+3-|\gamma|-|\rho|)} f_{\gamma, \rho}(2^N t_{j+1})$$
where we have used Lemma 3.5. Therefore,
$$\sum_{N=0}^\infty ||(\bar{\delta}^\mu \delta^\nu m)\chi_N||_{HS}||\chi_N(\bar{\delta}^\gamma \delta^\rho S_j )A_r^*||_{HS}$$
$$\leq C~ t_{j+1}\sum_{N=0}^\infty 2^{\frac{N}{2}(2n+3-l)}f_{\gamma, \rho}(2^Nt_{j+1}) \leq C~ t_{j+1}^{-n+\frac{l}{2}-1/2}.$$\\

Consequently, we have proved, by taking $l=2n$ and $l=2n+1$, the
estimates $|Z_r k_j(z)|\leq C~ t_{j+1}^{-1/2}|z|^{-2n}$ and  $|Z_r
k_j(z)|\leq C~ |z|^{-2n-1}$ and combining them we obtain the
estimate
$$|Z_r k_j(z)|\leq C~ t_{j+1}^{-1/4}|z|^{-2n-1/2}$$
We also have the estimates $|\bar{z_r}k_j(z)|\leq C~
|z||z|^{-2n-2}= C |z|^{-2n-1}$ and $|z_r k_j(z)|\leq C~ |z|
t_{j+1}^{1/2}|z|^{-2n-1}.$ Putting all these estimates together we
get
$$|\frac{\partial}{\partial z_r} k_j(z)|\leq C~ t_{j+1}^{1/4}|z|^{-2n-1/2}\leq C~ t_{j+1}^{-1/4}|z|^{-2n-1/2}.$$
Similarly we can prove $|\frac{\partial}{\partial \bar{z_r}}
k_j(z)|\leq C~ t_{j+1}^{-1/4}|z|^{-2n-1/2}$
for $r=1,2,...,n.$\\

Finally, we are ready to prove (2) of Proposition 3.6. When
$|z|>2|u|,~ |z-u|> (1/2) |z|$ and so
$$ |k_j(z-u)e^{-\frac{i}{2} \Im(z\cdot \bar{u})}-k_j(z)| \leq |k_j(z-u)|+|k_j(z)| $$
which is bounded by
$$ C~ t_{j+1}^{1/4}(|z-u|^{-2n-1/2}+|z|^{-2n-1/2})$$
$$\leq C~ t_{j+1}^{-1/4} |z|^{-2n-1/2} =  C~
\frac{|u|^{1/2}}{|z|^{2n+1/2}}\frac{t_{j+1}^{1/4}}{|u|^{1/2}}.$$
On the other hand $ |k_j(z-u)e^{-\frac{i}{2} \Im(z\cdot
\bar{u})}-k_j(z)|$ is bounded by
$$ |k_j(z-u)-k_j(z)|+|k_j(z)( e^{-\frac{i}{2} \Im(z\cdot \bar{u})}-1)|
\leq |u| |\nabla k_j(\tilde{z})|+|u||z||k_j(z)|$$ where
$\tilde{z}$ is a point on the line segment joining $(z-u)$ and
$z$. The gradient term gives the estimate
$$|u|~  |\nabla k_j(z)|\leq C~ |u| t_{j+1}^{-1/4} |\tilde{z}|^{-2n-1/2} \leq C~ |u| t_{j+1}^{-1/4} |z|^{-2n-1/2}$$
and  the other term is estimated by
$$|u||z||k_j(z)|\leq C~ |u| t_{j+1}^{1/4}|z|^{-2n-1/2}\leq C~ |u| t_{j+1}^{-1/4}|z|^{-2n-1/2}$$
which follows from $|k_j(z)|\leq C~ |z|^{-2n-2}$ and $|k_j(z)|\leq
C~ t_{j+1}^{1/2}|z|^{-2n-1}$. Thus
$$ |k_j(z-u)e^{-\frac{i}{2} \Im(z\cdot \bar{u})}-k_j(z)|
\leq C~ |u| t_{j+1}^{-1/4} |z|^{-2n-1/2} $$ which we write as $ C~
\frac{|u|^{1/2}}{|z|^{2n+1/2}}\frac{|u|^{1/2}}{t_{j+1}^{1/4}}.$
Combining the two estimates $ C~
\frac{|u|^{1/2}}{|z|^{2n+1/2}}\frac{t_{j+1}^{1/4}}{|u|^{1/2}}$ and
$ C~
\frac{|u|^{1/2}}{|z|^{2n+1/2}}\frac{|u|^{1/2}}{t_{j+1}^{1/4}}$ we
obtain
$$
|k_j(z-u)e^{-\frac{i}{2}\Im(z\cdot \bar{u})}-k_j(u)|  \leq C~
\frac{|u|^{1/2}}{|z|^{2n+1/2}}\min(\frac{t_{j+1}^{1/4}}{|u|^{1/2}},\frac{|u|^{1/2}}{t_{j+1}^{1/4}}).$$
This completes the proof of Proposition 3.6 for all $j\geq 1$. The
case $j=0$ is even simpler since $m_0= e^{nt_1}
e^{-t_1H}=\sum_{k=0}^\infty e^{2k t_1}(mP_k).$ It is estimated in
a similar fashion and we  leave the
details to the reader.\\

We will now prove the following result concerning the commutator
of $T^1_m$ with multiplication by a BMO function.
\begin{thm}
Let $m\in B(L^2(\R^n))$ be of class $C^{n+1}$ and satisfy
Mauceri's condition $(M_{2n+2})$. If $b\in BMO(\C^n)$, then there
exists a constant $C= C(p, m,n)$ such that
$$||[b, T^1_m]f||_p\leq C ||b||_*||f||_p$$
for $1<p<\infty$.
\end{thm}

\begin{proof}
The main step in the proof of the above theorem is the following
estimate: \be\tilde{M}^\sharp ([b, T^1_m]f)(z)\leq C
||b||_*(M_rT^1_mf(z)+M_{rs}f(z)) \ee where $r,s>1$ be such that
$1<rs<p$. If we can show that (3.6) is true, then the proof of the
theorem is immediate. As the proof of the theorem is similar to
the  Lemma 11 in \cite{J}, we  leave the details to the reader.

\end{proof}

We now turn our attention towards a proof of Theorem 1.2. In order to prove this theorem we need the following $ L^2 $ version of Theorem 3.1.

\begin{thm}
Consider the operator $T=k\times f$ where $k\in L^2(\C^n)$
satisfies the estimate \be\left( \int_{|z| > 2|u|}
|z|^{2n+2\delta}|k(z-u) e^{-\frac{i}{2}\Im (z.
\bar{u})}-k(z)|^2  dz \right)^{\frac{1}{2}}  \leq C~ |u|^{\delta}\ee for
some $\delta>0$. Then $Tf$ is bounded on $L^p(\C^n)$, $1<
p<\infty$. Moreover, if $Tf\in L^p(\C^n,w)$ for a dense class of
functions in $L^p(\C^n,w)$, $w\in A_{p/2}(\C^n)$, $2<p<\infty$,
then
$$\int_{\C^n}|Tf|^p w(z)dz\leq \int_{\C^n}|f(z)|^p w(z)dz$$
for all $w\in A_{p/2}$, $f\in L^p(\C^n, w)$.
\end{thm}

In order to prove the theorem we need the following lemma.
\begin{lem}
Let $k\in L^2(\C^n)$ satisfies the condition of the above theorem.
Then
$$\tilde{M}^{\sharp}(Tf)(v)\leq C ~ M_2f(v).$$

\end{lem}
\begin{proof}
Let $v\in \C^n$ and $Q$ be a cube containing it. Let
$u$ be the  center of $ Q$. Also, let $f_1= f\chi_{2Q}$ and $f_2= f-f_1$.
To prove the lemma it is enough to show that
$$\frac{1}{|Q|}\int_Q |Tf(z)e^{-\frac{i}{2}\Im (z.\bar{u})}-Tf_2(u)|dz\leq C~ M_2f(v).$$
The left hand side can be dominated by
$$\frac{1}{|Q|}\int_Q |Tf_1(z)|dz+\frac{1}{|Q|}\int_Q |Tf_2(z)e^{-\frac{i}{2}\Im (z.\bar{u})}-Tf_2(u)|dz .$$
The first term is easy to handle. Indeed, it can be estimated by
$$\left(\frac{1}{|Q|}\int_Q |Tf_1(z)|^2dz\right)^{\frac{1}{2}}.$$
Using the $L^2$ boundedness of $T$ we can dominate the above term
by
$$C~\left(\frac{1}{|Q|}\int_{2Q} |f(z)|^2dz\right)^{\frac{1}{2}}\leq 2^n C M_2f(v).$$
In order to estimate the second term we use the kernel estimate
given in the hypothesis. Using the definitions of $T$ and $f_1$,
we get
$$\frac{1}{|Q|}\int_Q \int_{\C^n\setminus 2Q}|k(z-w)e^{\frac{i}{2}\Im (z.\bar{w}-z \bar{u})}-k(u-w)e^{\frac{i}{2}\Im(u.\bar{w})}||f(w)|dw dz.$$
By H\"{o}lder's inequality the inner integral is dominated by the product of

$$ \left(\int_{\C^n}|u-w|^{2n+\delta}|k(u-w-u+z)e^{-\frac{i}{2}\Im((u-w)(\bar{u}-\bar{z}))}-k(u-w)|^2dw\right)^{\frac{1}{2}}$$
and
$$\left(\int_{\C^n\setminus
2Q}\frac{|f(w)|^2}{|u-w|^{2n+2\delta}}dw\right)^{\frac{1}{2}}.$$
Using the hypothesis of the lemma, we can observe that the first
integral is bounded by $|u-z|^\delta$ which further can be
dominated by $l(Q)^\delta$. The second integral is dominated by
$$\sum_{k=1}^\infty \int_{2^k l(Q)\leq |u-w|< 2^{k+1}l(Q)}\frac{|f(w)|^2}{|u-w|^{2n+2\delta}}.$$
One can easily see that the above sum is bounded by
$$l(Q)^{-\delta}M_2f(v).$$
Taking average over $ Q $ the lemma is proved.

\end{proof}

Now we are ready to prove Theorem 3.9. From (3.7) we can easily
deduce that the kernel of $T$ satisfies the following estimate
$$\int_{|z| > 2|u|}
|k(z-u) e^{-\frac{i}{2}\Im (z. \bar{u})}-k(z)|dz< C .$$ Hence
from Theorem 3.2 of \cite{M} we can conclude that $T$ is bounded
on $L^p(\C^n)$ for $1<p\leq 2$. For $p>2$, We will use the above
lemma. The point-wise estimate $|T f(z)|\leq M_d Tf(z)$ gives us
$$\int_{\C^n}|Tf(z)|^pdz\leq C~\int_{\C^n} (M_d (Tf)(z))^pdz.$$
As $T$ is bounded on $L^2(\C^n),$ we can use  Lemma 3.3 to conclude that
$$\int_{\C^n}|T f(z)|^pdz\leq C~ \int_{\C^n}|\tilde{M}^\sharp(Tf)(z)|^pdz$$
for any $p>2$. Now using Lemma 3.10 and the boundedness of $M_2$
on $L^p(\C^n)$, $p>2$ one can easily see
$$\int_{\C^n} |Tf(z)|^pdz\leq C~ \int_{\C^n} |f(z)|^pdz$$
As Lemma 3.3 is true for any $w\in A_p$, $1<p<\infty$, the remaining
part of the lemma can be proved by same arguments once we have a dense class of functions appearing in the hypothesis of the theorem.\\

The existence of a dense class of functions $\mathcal{D}\subset
L^p(\C^n,w)$ such that $Tf\in L^p(\C^n,w)$, $2<p<\infty$, is
guaranteed once assume the  estimate  \be\int_{\C^n}|z|^{2n+2\theta}|k(z)|^{2n+2 \theta}dz < C\ee
for some $\theta>0 $ on the kernel.
To see this let us consider the space $\mathcal{D}$ of all
smooth functions with compact support. Suppose $f$ is a function
in $\mathcal{D}$ whose support is contained in $B(0,R),$ the ball of
radius $R $ centered at the origin for some $R>0$.
Now, for $\epsilon>0$, using H\"{o}lder's inequality we see that
$\int_{|z|<2R}|Tf(z)|^pw(z) dz $ is bounded by
$$ \left(\int_{|z|<2R}w(z)^{1+\epsilon}dz\right)^{\frac{1}{1+\epsilon}}
\left(\int_{|z|<2R}|Tf(z)|^{p(1+\epsilon)/\epsilon}
\right)^{\frac{\epsilon}{1+\epsilon}}.$$ By the reverse H\"{o}lder
inequality, we can choose $\epsilon>0$ such that the first
integral is finite. The second integral is finite since $Tf\in
L^q$, $2<q<\infty$.\\

For $|z| > 2 R$,  applying H\"{o}lder's inequality in the definition of $ Tf(z) $ we get

$$|Tf(z)|\leq \left(\int_{\C^n}|z-w|^{2n+2\theta}|k(z-w)|^2dw\right)^{\frac{1}{2}}
\left(\int_{|w|<R}\frac{|f(w)|^2}{|z-w|^{2n+2\theta}} dw\right)^{\frac{1}{2}}.$$
By (3.8), the right hand side can be dominated by
$ C_R \|f\|_\infty |z|^{-n-\theta}.$
The above discussion leads us to the following estimate: \Bea
\int_{|z|>2R} |Tf(z)|^p w(z)dz&\leq& C
\sum_{j=1}^{\infty}\int_{2^j
R<|z| \leq 2^{j+1}R}\frac{w(z)}{|z|^{(n+\theta)p}}dz\\ &\leq&
C~\sum_{j=1}^{\infty}(2^jR)^{-p(n+\theta)}w(B(0, 2^{j+1}R)). \Eea
As $w\in A_{p/2}$, $w(B(0, 2^{j+1}R))$ is bounded by $ C (2^j
R)^{np}$, which implies the above sum is finite. Hence, $Tf\in
L^p(\C^n, w)$ for all $f\in \mathcal{D}$.\\

In view of the above observations, we can restate  Theorem 3.9 as follows.\\
\begin{thm}
Let us consider the operator $T=k\times f $ where the kernel $k\in
L^2(\C^n)$ satisfies the hypothesis (3.7) and
$$\left( \int_{\C^n}|z|^{2n+2\theta} |k(z)|^2 dz \right)^{\frac{1}{2}}\leq C$$
for some $ \delta >0,  \theta>0.$ Then $T$ is bounded on $L^p(\C^n)$,
$1<p<\infty$. Also, $T$ satisfies the following weighted norm
inequality
$$\int_{\C^n}|Tf(z)|^pw(z)dz\leq C~ \int_{\C^n}|f(z)|^p w(z)dz$$
where $f\in L^p(\C^n, w)$, $w\in A_{p/2}$, $2<p<\infty$.
\end{thm}
Now we are ready to prove Theorem 1.2. From the above discussions
we only need to prove the following theorem.

\begin{thm}
Let $m\in B(L^2(\R^n))$ be of class $C^{n+2}$ and satisfies
Mauceri's condition $(M_{n+2})$. Then
\begin{enumerate}
\item $\int_{\C^n} |z|^{2n+1}|k(z)|^2dz<C$
\item $\left( \int_{|z| > 2|u|} |z|^{2n+1}|k(z-u)
e^{-\frac{i}{2}\Im (z. \bar{u})}-k(z)|^2 dz
\right)^{\frac{1}{2}}\leq C~ |u|^{\frac{1}{2}}.$
\end{enumerate}
\end{thm}
In order to to prove the above theorem we need the following $L^2$
analogue of Proposition of 3.6. Once we have the following estimates, we immediately get the theorem since the series
$$ \sum_{j=1}^\infty  \min\left(\frac{t_{j+1}^{\frac{1}{4}}}{|u|^{\frac{1}{2}}},\frac{|u|^{\frac{1}{2}}}{t_{j+1}^{\frac{1}{4}}}\right) $$ converges.
Thus we are left with proving the following proposition.

\begin{prop}
For each $j=0, 1, 2,\cdots $ we have the estimates:
\begin{enumerate}
\item $\left(\int_{\C^n}
|z|^{2n+1}|k_j(z)|^2\right)^{\frac{1}{2}}dz \leq C~
t_{j+1}^{\frac{1}{4}}$ \item $\left( \int_{|z| > 2|u|}
|z|^{2n+1}|k_j(z-u) e^{-\frac{i}{2}\Im (z. \bar{u})}-k_j(z)|^2 dz
\right)^{\frac{1}{2}}\\ \leq C~
|u|^{\frac{1}{2}}\min\left(\frac{t_{j+1}^{\frac{1}{4}}}{|u|^{\frac{1}{2}}},\frac{|u|^{\frac{1}{2}}}{t_{j+1}^{\frac{1}{4}}}\right).$
\end{enumerate}
\end{prop}
\begin{proof}
To prove (1) we claim that
$$\left(\int_{\C^n}
|z|^{2l}|k_j(z)|^2\right)^{\frac{1}{2}}dz \leq C~ t_{j+1}^{(l-n)/2}
$$
whenever $l\leq n+1$. In order to estimate the $L^2$ norm of
$|z|^l|k_j(z)|$, it is enough to estimate the $L^2$ norm of
$z^\alpha \bar{z}^\beta k_j(z)$, for $|\alpha|+|\beta|=l$. That
is, we have to estimate the Hilbert-Schmidt norm of $\delta^\alpha
\bar{\delta}^\beta (MS_j)$. As we have done in Proposition 3.6, it
is enough to estimate
$$\sum_{N=0}^\infty (\delta^\mu \bar{\delta}^\nu m)\chi_N. \chi_N(\delta^\gamma \bar{\delta}^\rho S_j)$$
where $|\mu|+|\nu|+|\gamma|+|\rho|=l$. The Hilbert-Schimdt norm of
the above sum is dominated by
$$\sum_{N=0}^\infty ||(\delta^\mu \bar{\delta}^\nu m)\chi_N||_{HS}. || \chi_N(\delta^\gamma \bar{\delta}^\rho S_j)||_{op}.$$
Now from Lemma 4.4 in \cite{M} we have
$$|| \chi_N(\delta^\gamma \bar{\delta}^\rho S_j)||_{op}\leq C~ t_{j+1} 2^{-N(|\gamma|+|\rho|-2)}f_{\gamma, \rho}(2^N t_{j+1}).$$
Hence using the above estimate and the hypothesis on $m$ one can
get
$$||\sum_{N=0}^\infty (\delta^\mu \bar{\delta}^\nu m)\chi_N. \chi_N(\delta^\gamma \bar{\delta}^\rho S_j)||_{HS}$$
$$\leq C~\sum_{N=0}^\infty 2^{\frac{N}{2}(n-|\mu|-|\nu|)}t_{j+1}2^{-N(|\gamma+|\rho|-2)/2}f_{\gamma, \rho}(2^N t_{j+1})$$
$$\leq C~ t_{j+1} \sum_{N=0}^\infty 2^{\frac{N}{2}(n-l+2)}f_{\gamma,\rho}(2^N t_{j+1}) \leq C~ t_{j+1}^{(l-n)/2}.$$
This proves our claim. Now, when $l=n$, $||z^n k_j(z)||_2\leq
 C$ and when $l=n+1$, $||z^{n+1}k_j(z)||_2\leq C~ t_{j+1}^{\frac{1}{2}}$.
Combining these two estimates we get
$$||z^{n+\frac{1}{2}}k_j(z)||_2\leq C~ t_{j+1}^{\frac{1}{4}}$$
which proves (1). Again from the above estimation we see that
$$||z^{n+2} K_j(z)||_2\leq C~ t_{j+1}\sum_{N=0}^\infty f_{\gamma, \rho}(2^N t_{j+1})\leq C~ \sum_{N=0}^\infty 2^N t_{j+1} f_{\gamma, \rho}(2^N t_{j+1})\leq C.$$

In order to prove (2) we need to estimate the gradient of $k_j$.
Earlier we have already noted  that
$$\frac{\partial}{\partial z_r}k_j=Z_r k_j-\frac{1}{4} \bar{z}_rk_j.$$
In order to estimate the $L^2$ norm of $|z|^l Z_r k_j$ it is enough to
estimate
$$\sum_{N=0}^\infty ||(\delta^\mu \bar{\delta}^\nu m)\chi_N||_{HS}||\chi_N(\delta^\gamma \bar{\delta}^\rho( S_jA_r^*))||_{op}.$$
From (4.4) in \cite{M} it is not difficult to see that
$$|| \chi_N(\delta^\gamma \bar{\delta}^\rho (S_j A_j^*))||_{op}\leq C~ t_{j+1} 2^{-\frac{N}{2}(|\gamma|+|\rho|-3)}f_{\gamma, \rho}(2^N t_{j+1}).$$
The above estimate and similar arguments used in the proof of (1) lead us
to the  estimate
$$||z^l Z_r k_j||_2\leq C~ t_{j+1}^{(l-n-1)/2}.$$
Putting $l=n$ and $l=n+1$ we get the estimates
$$||z^n Z_r k_j||_2\leq C~ t_{j+1}^{-\frac{1}{2}},  ~~~~~~||z^{n+1} Z_r k_j||_2\leq C$$
respectively. Hence
$$||z^{n+\frac{1}{2}} Z_r k_j||_2\leq C~ t_{j+1}^{-\frac{1}{4}}.$$
For $z_r k_j(z)$ one can see that
$$||z^n z_r k_j||_2\leq C~ ||z^{n+1} k_j||_2\leq C $$
and
$$||z^{n+1} z_r k_j||_2\leq C~ ||z^{n+2} k_j||_2\leq C.$$
Thus we have $||z^{n+\frac{1}{2}} z_r k_j||_2\leq C~. $
This proves the estimate
$$||z^{n+\frac{1}{2}}\nabla k_j(z)||_2\leq C~t_{j+1}^{-\frac{1}{4}}.$$\\

Finally, if $|z|>2|u|$, then it follows that $|z-u|>\frac{1}{2}|z|$. By triangle inequality,
$|k_j(z-u)
e^{-\frac{i}{2}\Im (z. \bar{u})}-k_j(z)|\leq
|k_j(z-u)|+|k_j(z)|$ and hence we have
\be \left( \int_{|z|>2|u|}
|z|^{2n+1}|k_j(z-u) e^{-\frac{i}{2}\Im (z. \bar{u})}-k_j(z)|^2
\right)^{\frac{1}{2}} \leq C~
|u|^{\frac{1}{2}}\frac{t_{j+1}^{\frac{1}{4}}}{|u|^{\frac{1}{2}}}.\ee
On the other hand, by mean value  theorem
$|k_j(z-u)
e^{-\frac{i}{2}\Im (z.
\bar{u})}-k_j(z)| $ is bounded by
 $$ |k_j(z-u)-k_j(z)|+|k_j(z)(e^{\frac{i}{2}\Im (z.
\bar{u})}-1)| \leq |u||\nabla
k_j(\tilde{z})|+|u||z||k_j(z)|$$
where $\tilde{z}$ is a point on the line segment
joining $(z-u)$ and $z$. Hence we get
\be\left( \int_{|z|>2|u|}
|z|^{2n+1}|k_j(z-u) e^{-\frac{\iota}{2}\Im (z. \bar{u})}-k_j(z)|^2
\right)^{\frac{1}{2}}\ee
$$ \leq C~|u|t_{j+1}^{-\frac{1}{4}}=C~
|u|^{\frac{1}{2}}\frac{|u|^{\frac{1}{2}}}{t_{j+1}^{\frac{1}{4}}}.$$
Comparing (3.9) and (3.10) we get the required result.
\end{proof}

\section [Fourier multipliers on the Heisenberg group $H^n$]
{Fourier multipliers on the Heisenberg group}

In this section we prove the theorems  stated in the introduction.
We begin with Theorem 1.7 which is very easy to prove. The proof
is based on the following lemma. Consider  convolution operators
$$ S(\lambda)
f(\xi) = \int_{\R^n} k_\lambda(\xi-\eta)f(\eta) d\eta $$ and
denote by $ S_2(\lambda) $ the following operator defined on
functions of $2n$ variables by
$$ S_2(\lambda)f(x,y) = \int_{\R^n} k_\lambda(\eta) f(x,y+\eta) d\eta.$$
We also let $ e_\lambda $ stand for the operator $ (e_\lambda
f)(x,y) = e^{(i/2)\lambda x\cdot y}f(x,y).$

\begin{lem} For every $ \lambda \in \R^* $ we have
$  T^\lambda_{S(\lambda)} = e_\lambda S_2(\lambda) e_{-\lambda}.$

\end{lem}
\begin{proof} The lemma is proved by simple calculation. We note that $$W_\lambda(e_\lambda
S_2(\lambda)e_{-\lambda}f)\phi(\xi)$$
$$=\int_{\C^n}e^{(i/2)\lambda x\cdot y} S_2(\lambda) e_{-\lambda}f(x,y) e^{i\lambda(x\cdot\xi+\frac{1}{2}x\cdot y)}\phi(\xi+y)dxdy$$
$$=\int_{\C^n}\int_{\R^n} k_\lambda(\eta)e^{(i/2)\lambda x\cdot (y+\eta)}f(x,y+\eta) e^{i\lambda(x\cdot\xi+x\cdot y)}\phi(\xi+y)dxdyd\eta$$
$$=\int_{\C^n}\int_{\R^n} k_\lambda(\eta)e^{(i/2)\lambda x\cdot y}f(x,y) e^{i\lambda \{x\cdot(\xi-\eta)+x\cdot y\}}\phi(\xi-\eta+y)dxdyd\eta$$
The last integral simplifies to give
$$\int_{\R^n} k_\lambda(\eta)\int_{\C^n} f(x,y)e^{i\lambda \{x\cdot(\xi-\eta)+x\cdot y\}}\phi(\xi-\eta+y)dxdyd\eta = $$
$$\int_{\R^n}k_\lambda(\eta)W_\lambda(f)\phi(\xi-\eta)d\eta =S(\lambda)W_\lambda(f)\phi(\xi).$$
Hence the lemma is proved.

\end{proof}

From the lemma we observe that $ T^\lambda_{S(\lambda)}$ is
bounded on $ L^p(\C^n) $ whenever $ S(\lambda) $ is bounded on $
L^p(\R^n).$ We also note that when $ f $ is a function on the
Heisenberg group, $ e_{-\lambda}f^\lambda(z) = (\tau( x \cdot
y)f)^\lambda(z) $ where $ \tau(a)f(z,t) =f(z,t+a/2).$ Consider the
multiplier transform
$$ T_Sf(z,t) = \frac{1}{2\pi} \int_{-\infty}^\infty e^{-i\lambda t} T^\lambda_{S(\lambda)}f^\lambda(z) d\lambda.$$
In view of the lemma and the above observation we see that
$$ T_Sf(z,t+\frac{1}{2}x\cdot y) =  \frac{1}{2\pi} \int_{-\infty}^\infty e^{-i\lambda t} S_2(\lambda)(\tau(x\cdot y)f)^\lambda(z) d\lambda.$$
This means that
$$ (\tau(x\cdot y)T_S\tau(-x\cdot y)f)(z,t)  =\frac{1}{2\pi} \int_{-\infty}^\infty e^{-i\lambda t}  S_2(\lambda)f^\lambda(z) d\lambda.$$
Under the hypothesis of Theorem 1.7 the families $ S(\lambda) $ and $ \lambda S'(\lambda) $ are both R-bounded.
Hence the same is true of $ S_2(\lambda) $ and consequently the right hand side of the above equation
defines a bounded operator  on $ L^p(\H^n).$ As translation in the last variable is a bounded operator on $ L^p(\H^n) $
Theorem 1.7 follows immediately.\\

Returning to general multiplers on the Heisenberg group  recall
that
$$T_m f(z,t)= (2\pi)^{-1}\int_{-\infty}^{\infty} e^{i \lambda t}
T^\lambda_{m(\lambda)} f^\lambda(z)d\lambda
$$
and the R-boundedness of $M(\lambda)=T^\lambda_m(\lambda)$ can be
proved now. By Lemma 2.1
$\delta^{-1}_{\sqrt{\lambda}}T^\lambda_{m(\lambda)}\delta_{\sqrt{\lambda}}=
T^1_{\tilde{m}(\lambda)}$ which means that
$\delta^{-1}_{\sqrt{\lambda}}T^\lambda_{m(\lambda)}\delta_{\sqrt{\lambda}}$
is a Weyl multiplier with multiplier
$\tilde{m}(\lambda)=\delta_{\sqrt{\lambda}}^{-1}m(\lambda)\delta_{\sqrt{\lambda}}.$
As $\lambda^{\frac{n}{2}} \delta_{\sqrt{\lambda}}$ is  unitary and
Hilbert-Schmidt operator norm is unitary invariant, Lemma 2.2
along with the hypothesis on $m(\lambda)$ stated in Theorem 1.8
allows us to conclude that $\tilde{m}(\lambda)$ satisfies
Mauceri's condition $ (M_{n+1})$ . Consequently, by Theorem 1.3
$T^1_{\tilde{m}(\lambda)}$ satisfies the weighted norm inequality
$$\int_{\C^n}|T^1_{\tilde{m}(\lambda)}f(z)|^p w(z)dz\leq C_w \int_{\C^n} |f(z)|^pw(z)dz,$$
where $C_w$ depends on $w$ but independent of $\lambda$. The above
gives the inequality
$$\int_{\C^n} |T^\lambda_{m(\lambda)}f(z)|^p w(\sqrt{\lambda} z)dz\leq C_w \int_{\C^n} |f(z)|^p w(\sqrt{\lambda}z)dz.$$
Since $w(\sqrt{\lambda z)}$ satisfies $A_p$ condition with the
same norm as $w$, it follows that
$$\int_{\C^n}|T^\lambda_{m(\lambda)}f(z)|^p w(z)dz\leq C_w \int_{\C^n}|f(z)|^pw(z)dz.$$
By the theorem of Rubio de Francia (see \cite{R}) we get
the R-boundedness of $M(\lambda)=T^\lambda_{m(\lambda)}$.\\

We now turn our attention to the R-boundedness of $\lambda
\frac{d}{d\lambda}M(\lambda)$. According to the Lemma 2.4 we need
to
treat three  families of operators.\\

\begin{prop}
Under the hypothesis of Theorem 1.9 the families
$T^\lambda_{[m(\lambda), \xi \cdot\nabla]}$ and
$T^\lambda_{2\lambda \frac{d}{d\lambda}m(\lambda)}$ are R-bounded.
\end{prop}
\begin{proof}
The proof of R-boundedness of $T^\lambda_{2\lambda
\frac{d}{d\lambda}m(\lambda)}$ is similar to that of
$T^\lambda_{m(\lambda)}$ as $\lambda \frac{d}{d\lambda}
m(\lambda)$ satisfies the same conditions as $m(\lambda)$. To
Treat the other one we write
\begin{eqnarray*}
4\lambda \xi_j \frac{\partial}{\partial
\xi_j}&=&(A_j(\lambda)+A_j^*(\lambda))(A_j(\lambda)-A_j^*(\lambda)\\
&=& A_j(\lambda)^2-A_j^*(\lambda)^2+[A_j^*(\lambda),
A_j(\lambda)].
\end{eqnarray*}
Since $[A_j^*(\lambda), A_j(\lambda)]=-2\lambda$ I we see that
$$4\lambda[m(\lambda), \xi_j \frac{\partial}{\partial \xi_j}]=[m(\lambda), A_j(\lambda)^2]-[m(\lambda), A_j^*(\lambda)^2]$$
which can be written as
\begin{eqnarray*}
4\lambda[m(\lambda), \xi_j \frac{\partial}{\partial
\xi_j}]=\sqrt{\lambda}(\delta_j(\lambda)m(\lambda))A_j(\lambda)+\sqrt{\lambda}A_j(\lambda)\delta_j(\lambda)m(\lambda)\\+
\sqrt{\lambda}(\bar{\delta_j}(\lambda)m(\lambda))A^*_j(\lambda)+\sqrt{\lambda}A^*_j(\lambda)\bar{\delta_j}(\lambda)m(\lambda)
\end{eqnarray*}
We just consider one family corresponding to the multiplier
$$\lambda^{-1/2}(\delta_j(\lambda) m(\lambda))A_j(\lambda)=m_j(\lambda).$$
Since $\delta_j(\lambda)$ and $\bar{\delta_j}(\lambda)$ are
derivations with $\delta_j(\lambda)A_j(\lambda)=0$ and
$\bar{\delta_j}(\lambda)A_j(\lambda)=2\lambda^{1/2}I$ it follows
that the above family satisfies condition $(M_{n+1}) $.
Consequently the operator family $T^\lambda_{m_j(\lambda)}$ is
R-bounded. The other families are treated in  the same way.
\end{proof}

Finally, we are left with the family $[B,
T^\lambda_{m(\lambda)}]$. Recalling that  $B=\sum_{j=1}^n(z_j
\frac{\partial}{\partial z_j}+\bar{z_j} \frac{\partial}{\partial
\bar{z_j}})$ we consider
$$[z_j
\frac{\partial}{\partial z_j}+\bar{z_j} \frac{\partial}{\partial
\bar{z_j}}, T^\lambda_{m(\lambda)}]=[z_j
Z_j(\lambda)+\bar{z_j}\bar{Z_j}(\lambda),
T^\lambda_{m(\lambda)}].$$ As $(z_j\frac{\partial}{\partial
z_j}+\bar{z_j}\frac{\partial}{\partial \bar{z_j}})$ commutes with
dilations we get
$$\delta^{-1}_{\sqrt{\lambda}}[z_j
\frac{\partial}{\partial z_j}+\bar{z_j} \frac{\partial}{\partial
\bar{z_j}}, T^\lambda_{m(\lambda)}] \delta_{\sqrt{\lambda}} =[z_j
Z_j(1)+\bar{z_j}\bar{Z_j}(1), T^1_{\tilde{m}(\lambda)}]$$ where we
have used the relations
$$\delta_{\sqrt{\lambda}}Z_j(1) \delta^{-1}_{\sqrt{\lambda}}=\lambda^{-1/2}Z_j(\lambda),~~~
\delta_{\sqrt{\lambda}}\bar{Z_j}(1)\delta^{-1}_{\sqrt{\lambda}}=\lambda^{-1/2}\bar{Z_j}(\lambda).$$
We observe the following relations:
$$W(z_j Z_j(1)T^1_{\tilde{m}(\lambda)}f)=-2\bar{\delta_j}(\tilde{m}(\lambda)W(f) A_j^*),$$
$$W(\bar{z_j} \bar{Z_j}(1)T^1_{\tilde{m}(\lambda)}f)=-2\delta_j(\tilde{m}(\lambda)W(f) A_j)$$
and consequently
$$W(T^1_{\tilde{m}(\lambda)} z_j Z_j(1) f)=- 2 i \tilde{m}(\lambda)\bar{\delta_j}(W(f)A_j^*),$$
$$W(T^1_{\tilde{m}(\lambda)} \bar{z}_j \bar{Z_j}(1) f)=- 2 i \tilde{m}(\lambda)\bar{\delta_j}(W(f)A_j).$$
Therefore, we have
$$W([z_j
\frac{\partial}{\partial z_j}+\bar{z_j} \frac{\partial}{\partial
\bar{z_j}}, T^1_{\tilde{m}(\lambda)}]f)=-2 \bar{\delta_j }
\tilde{m}(\lambda)W(f)A_j^*- 2 \delta_j\tilde{m}(\lambda) W(f)
A_j.
$$

The above relation clearly  shows that $[B,
T^\lambda_{m(\lambda)}]$ is not a Weyl multiplier, since $W(f)$
need not commute with $A_j$ and $A_j^*$ in general. Consequently,
we cannot hope to show that $[B, T^\lambda_{m(\lambda)}]$ is
R-bounded. Indeed, when $m(\lambda)= A_j(\lambda)
H(\lambda)^{-1/2}$ which corresponds to the Riesz transforms $R_j$
on $H^n$, the operator $[B, T^\lambda_{m(\lambda)}]$ is not even
bounded on $L^2(\C^n)$. In
spite of this we can easily prove\\
\begin{prop} The family $[B, T^\lambda_{m(\lambda)}]
L_\lambda^{-1/2}, \lambda\in \R^*$ is R-bounded on $L^p(\C^n),
1<p<\infty$.
\end{prop}

\begin{proof}
Since $W_\lambda(L_\lambda^{-1/2} f )= W_\lambda(f)
H(\lambda)^{-1/2}$, we only need to show that the operators
$\bar{\delta_j}\tilde{m}(\lambda)Z_j(1) L^{-1/2}_1$ and $
\delta_j\tilde{m}(\lambda)\bar{Z_j}(1) L^{-1/2}_1$ satisfies
weighted norm inequalities on $L^p(\C^n)$ for any weight $w\in
A_p(\C^n).$ But $Z_j(1) L_1^{-1/2}$ and $\bar{Z_j}(1) L_1^{-1/2}$
are oscillatory singular integral operators and hence satisfy
weighted norm inequalities according to the theorem of Lu-Zhang
\cite{LZ}. Since $\bar{\delta_j} \tilde{m}(\lambda)$ and $\delta_j
\tilde{m}(\lambda)$ satisfy the condition $(M_{n+1})$ they define
the Weyl multipliers which satisfy weighted norm inequalities.
This proves the proposition
\end{proof}

Combining Propositions 4.2 and 4.3 and using the fact that
$\widehat{(\mathcal{L}^{-1/2} f)}(\lambda)=\hat{f}(\lambda)
H(\lambda)^{-1/2}$ we obtain Theorem 1.9. In order to prove
Theorem 1.10 we make use of the following observation. When $f$ is
a polyradial function  $ W(f)$ commutes with
$H_j=-\frac{\partial^2}{\partial \xi_j^2} +\xi^2_j, j=1,2,..., n$.
In view of this we have
$$2 \bar{\delta_j} \tilde{m}(\lambda) W(f) A_j^*= 2 \bar{\delta_j}\tilde{m}(\lambda) H_j^{1/2} W(f) H^{-1/2} A_j^*$$
which can be written as (since, $H_j=\frac{1}{2}( A_j A_j^*+A_j^*
A_j)$)
$$\bar{\delta_j} \tilde{m}(\lambda) A_j A_j^* H^{-1/2}W(f) H^{-1/2} A_j^*+
\bar{\delta_j} \tilde{m}(\lambda) A_j^* A_j H^{-1/2}W(f) H^{-1/2}
A_j^*.$$ The operator families  $\bar{\delta_j} \tilde{m}(\lambda)
A_j$ and $\bar{\delta_j} \tilde{m}(\lambda) A_j^*$ satisfy the
condition $(M_{n+1})$  and $A_j^* H^{-1/2}$, $A_j H^{-1/2}$ and $
H^{-1/2}A_j^*$ define oscillatory singular integrals (being
variations of Riesz
transforms).\\

Let $T(n)\subset U(n)$ be the torus which acts on $\C^n$ by
$\rho(\sigma)f(z)=f(e^{i\theta_1} z_1, ...,e^{i\theta_n} z_n)$ if
$\sigma$ is the diagonal matrix with entries $e^{i\theta_1}, ...,
e^{i\theta_n}$ then
$$Rf(z)=\int_{T(n)} \rho(\sigma)f(z) d\sigma$$
is polyradial and $||Rf||_p\leq ||f||_p$. Thus for every $ w\in
A_p(\C^n)$ which is polyradial, the operators $R \circ [B,
T^1_{\tilde{m}(\lambda)}] \circ R$ satisfy weighted norm inequalities. By
a theorem of Duandikoetxea et al \cite{D1}, a polyradial function
$w$ belongs to $A_p(\C^n)$ if and only if $w(r_1,..., r_n)$
belongs to $A_p(\R^n_+, d\mu)$ where $d\mu(r)=\Pi_{j=1}^n r_j
dr_j$. Consequently the families $R\circ T^\lambda_m(\lambda)\circ R$ and
$\lambda \frac{d}{d\lambda}R\circ T^\lambda_{m(\lambda)}\circ R$ are
R-bounded on $L^p(\R^n_+, d\mu)$.
This proves that $R\circ T_m\circ R$ is bounded on $L^p(H^n)$.\\

Finally, coming to the proof of the Theorem 1.6 recall that the
Riesz transforms which correspond to the multipliers
$m(\lambda)=A_j(\lambda) H(\lambda)^{-1/2}$ are bounded on
$L^p(H^n), 1<p<\infty$, see \cite{CG}. The theorem will be proved
if we show that $[B, T^1_{\tilde{m}(\lambda)}]$ is not bounded on
$L^2(\C^n)$. Note that $\tilde{m}(\lambda)=A_j H^{-1/2}$ and we
have to show that the Hilbert-Schmidt norm of
$$S= \bar{\delta_j}(A_j H^{-1/2})W(f) A_j^*+\delta_j(A_j H^{-1/2})W(f) A_j$$
is not bounded by a constant multiple of $||f||_2$. It can be
easily seen that
$\bar{\delta_j}H^{-1/2}=((H-2)^{-1/2}-H^{-1/2})A_j^*$ and
 $\delta_jH^{-1/2}=((H+2)^{-1/2}-H^{-1/2})A_j.$
When we take $f=\bar{\Phi }_{\alpha \beta}$ then
$W(f)\varphi=(\varphi, \Phi_\alpha)\Phi_\beta$ and therefore
$S\varphi_\mu$ survives only when $\mu=\alpha+e_j$ or
$\mu=\alpha-e_j$ where $ e_j $ is the $ j$-th coordinate vector.
Moreover, $S\Phi_{\alpha+e_j}=(2\alpha_j+2)^{1/2}
\delta_j(A_jH^{-1/2}) \Phi_\beta$ and
$S\Phi_{\alpha-e_j}=(2\alpha_j)^{1/2}\bar{\delta_j}(A_jH^{-1/2})
\Phi_\beta.$ This shows that
$$||S||^2_{HS}=(2\alpha_j+2)||\delta_j(A_j H^{-1/2})\Phi_\beta||^2_2+(2\alpha_j)||\bar{\delta_j}(A_jH^{-1/2}) \Phi_\beta||^2_2.$$
Since $||f||_2=1$ it is clear that $||S||_{HS}\leq C~ ||f||_2$
cannot
 be satisfied. This proves Theorem 1.6.

\newpage
\begin{center}
{\bf Acknowledgments}
\end{center}
The first author is thankful to CSIR, India, for the financial
support. The work of the second author is supported by J. C. Bose
Fellowship from the Department of Science and Technology (DST) and
also by a grant from UGC via DSA-SAP. Both authors thank the referee for his careful reading  and useful comments which were used in revising the manuscript.


\begin{thebibliography}{10}
\bibitem{BT} S. Bagchi and S. Thangavelu, \textit{On Hermite pseudo-multipliers }, arXiv:1311.5399

\bibitem{CG} M. Christ and D. Geller, \textit{Singular integral characterizations of Hardy spaces on homogeneous groups}, Duke Math. J. 51 (1984), no. 3, 547-598.

 \bibitem{D} J. Duoandikoetxea, \textit{Fourier analysis}, Graduate Studies in Mathematics, 29. American Mathematical Society, Providence, RI, 2001
\bibitem{D1}J. Duoandikoetxea, A. Moyua, O. Oruetxebarria and E. Seijo, \textit{Radial Ap weights with
applications to the disc multiplier and the Bochner-Riesz
operators}, Indiana Univ. Math. J. 57 (2008), no. 3, 1261-1281.
\bibitem{F} G. Folland, \textit{Harmonic analysis in phase space}, Annals of Mathematics
Studies,122, Princeton University Press, Princeton, NJ, 1989.
\bibitem{H} W. Hebisch, \textit{Multiplier theorem on generalized Heisenberg groups}, Colloq. Math. 65 (1993), no. 2, 231-239

\bibitem{J} S. Janson, \textit{Mean oscillation and commutators of singular integrals
operators}, Ark.Mat, 16(1978),263-270

\bibitem{JST} K. Jotsaroop, P. K. Sanjay and S. Thangavelu, \textit{Riesz transforms and multipliers for the Grushin
operator}, J. Anal. Math., 119(2013), 255-273

\bibitem{L} C. Lin, \textit{$L^p$ multipliers and their $H^1-L^1$ estimates on the Heisenberg
group}, Rev Mat Iberoamericana. Vol 11, N. 2, (1995).

\bibitem{LZ} S. Lu and Y. Zhang, \textit{The weighted norm inequality for a class of oscillatory integral
operators}, Chinese Science Bulletin, 1992, 37(1): 9-13.


\bibitem{Ma} A. Martini and A. Sikora, \textit{Weighted Plancherel estimates and sharp spectral multipliers for the Grushin
operators}, Math. Res. Lett. 19 (2012), no. 5, 1075-1088.


\bibitem{M} G. Mauceri, \textit{The Weyl transform and bounded operators
on $L^p(\R^n)$},  J. Funct. Anal., 39 (1980), no. 3, 408-429.
\bibitem{M2} G. Mauceri, \textit{Zonal multipliers on the Heisenberg
group}, Pacific J. Math, 95 (1981), no. 1, 143-159.

\bibitem{MM} G. Mauceri and L. de Michele, \textit{multipliers on the Heisenberg group}, Michigan Math. J., 26 (1979), no. 3, 361-371,

\bibitem{MS} D. Muller and E. M. Stein,  \textit{On spectral multipliers for Heisenberg and related groups}, J. Math. Pures Appl. (9) 73 (1994), no. 4, 413-440.


\bibitem{R}  J. L. Rubio de Francia, \textit{Vector-valued inequalities for operators in $L^p$ spaces}. Bull. London Math. Soc. 12 (1980), no. 3, 211-215.




\bibitem{T1} S. Thangavelu, \textit{Harmonic analysis on the Heisenberg group}, Progress in Mathematics, 159, Birkhauser Boston, Inc., Boston, MA, 1998.

\bibitem{W} L. Weis, \textit{Operator valued Fourier multiplier theorems and maximal $L^p$
regularity}, Math. Ann. 319(2001), 735-758.


\end{thebibliography}
\end{document}